\documentclass[10pt]{amsart}

\usepackage{amsmath}
\usepackage{amsthm}
\usepackage{amssymb}
\usepackage[all]{xy}
\usepackage{color}
\usepackage{graphicx}

\theoremstyle{plain}
\newtheorem{lemma}[subsection]{Lemma}
\newtheorem{proposition}[subsection]{Proposition}
\newtheorem{theorem}[subsection]{Theorem}
\newtheorem{corollary}[subsection]{Corollary}
\theoremstyle{definition}
\newtheorem{definition}[subsection]{Definition}
\newtheorem{example}[subsection]{Example}
\newtheorem{remark}[subsection]{Remark}

\numberwithin{equation}{section}

\DeclareMathOperator*{\colim}{colim}
\DeclareMathOperator*{\hocolim}{hocolim}

\newcommand{\sA}{\mathsf A}
\newcommand{\B}{\mathfrak{B}}
\newcommand{\BrCat}{\text{$\Br$-$\Cat$}}
\newcommand{\BrCatB}{\text{$\Br$-$\Cat^{\B}$}}
\newcommand{\bld}{\mathbf}
\newcommand{\br}{\mathfrak{b}}
\newcommand{\co}{\colon\thinspace}
\newcommand{\down}{\!\downarrow\!}
\newcommand{\ev}{\mathrm{ev}}
\newcommand{\id}{\mathrm{id}}
\newcommand{\hB}{_{h\mathfrak{B}}}

\newcommand{\sM}{\mathsf M}
\newcommand{\op}{\mathrm{op}}
\newcommand{\ra}{\rightarrow}
\newcommand{\SB}{{\mathcal{S}^\mathfrak{B}}}

\newcommand{\xr}{\xrightarrow}
\newcommand{\xl}{\xleftarrow}

\newcommand{\Cat}{\mathit{Cat}}
\newcommand{\Sym}{\mathsf{Sym}}
\newcommand{\Sp}{\mathit{Sp}}
\newcommand{\Br}{\mathsf{Br}}
\newcommand{\nerve}{\mathrm{N}}

\newcommand{\BBox}{\mathrm{B}^{^{_\boxtimes}}}

\newcommand{\cA}{\mathcal A}
\newcommand{\cB}{\mathcal B}
\newcommand{\cC}{\mathcal C}
\newcommand{\cD}{\mathcal D}
\newcommand{\cE}{\mathcal E}
\newcommand{\cI}{\mathcal I}
\newcommand{\cK}{\mathcal K}
\newcommand{\cM}{\mathcal M}
\newcommand{\cP}{\mathcal P}
\newcommand{\cS}{\mathcal S}
\newcommand{\cV}{\mathcal V}
\newcommand{\cW}{\mathcal W}

\title{Braided injections and double loop spaces}

\date{\today}

\author{Christian Schlichtkrull} \address{Christian Schlichtkrull,
    Department of Mathematics, University of Bergen, P.O. Box 7800, N-5020 Bergen, Norway} \email{christian.schlichtkrull@math.uib.no}

\author{Mirjam Solberg} \address{Mirjam Solberg,
    Department of Mathematics, University of Bergen, P.O. Box 7800, N-5020 Bergen, Norway} \email{mirjam.solberg@math.uib.no}

\begin{document}
\begin{abstract}
We consider a framework for representing double loop spaces (and more generally $E_2$ spaces) as commutative monoids. There are analogous commutative rectifications of braided monoidal structures and we use this framework to define iterated double deloopings. We also consider commutative rectifications of $E_{\infty}$ spaces and symmetric monoidal categories and we relate this to the category of symmetric spectra. 
\end{abstract}

\maketitle

\section{Introduction}
The study of multiplicative structures on spaces has a long history in algebraic topology. For many spaces of interest the notion of a strictly associative and commutative multiplication is too rigid and must be replaced by the more flexible notion of an $E_{\infty}$ multiplication encoding higher homotopies between iterated products. This is analogous to the situation for categories where strictly commutative multiplications  rarely occur in practice and  the more useful $E_{\infty}$ notion is that of a symmetric monoidal structure. Similar remarks apply to multiplicative structures on other types of objects. However, for certain kinds of applications it is desirable to be able to replace $E_{\infty}$ structures by strictly commutative ones, and this can sometimes be achieved by modifying the underlying category of objects under consideration. An example of this is the introduction of modern categories of spectra (in the sense of stable homotopy theory) \cite{EKMM,HSS,MMSS} equipped with symmetric monoidal smash products. These categories of spectra have homotopy categories equivalent to the usual stable homotopy category but come with refined multiplicative structures allowing the rectification of $E_{\infty}$ ring spectra to strictly commutative ring spectra. This has proven useful for the import of ideas and constructions from commutative algebra into stable homotopy theory. Likewise there are symmetric monoidal refinements of spaces \cite{BCS,Sagave-Schlichtkrull} allowing for analogous rectifications of $E_{\infty}$ structures. 

Our main objective in this paper is to construct similar commutative rectifications in braided monoidal contexts. In order to provide a setting for this we introduce the category $\B$ of \emph{braided injections}, see Section \ref{sec:braided-injections}. This is a braided monoidal small category that relates to the category $\cI$ of finite sets and injections in the same way the braid groups relate to the symmetric groups. 
We first explain how our rectification works in the setting of small categories $\Cat$ and let $\Br$-$\Cat$ denote the category of braided (strict) monoidal small categories.
Let $\Cat^{\B}$ be the diagram category of functors from $\B$ to $\Cat$ and let us refer to such functors as 
\emph{$\B$-categories}. The category $\Cat^{\B}$ inherits a braided monoidal convolution product from $\B$ and there is a corresponding category $\Br$-$\Cat^{\B}$ of braided monoidal $\B$-categories. A morphism $A\to A'$ in $\Br$-$\Cat^{\B}$ is said to be a \emph{$\B$-equivalence} if the induced functor of Grothendieck constructions $\B\!\int\!A\to \B\!\int\!A'$ is a weak equivalence of categories in the usual sense. We write $w_{\B}$ for the class of $\B$-equivalences and $w$ for the class of morphisms in $\Br$-$\Cat$ whose underlying functors are weak equivalences. The following rectification theorem is obtained by combining Proposition~\ref{prop:Bint-Delta--braided-equivalence} and Theorem~\ref{thm:Braided-B-Cat-rectification}. 

\begin{theorem}\label{thm:main-theorem-1}
The Grothendieck construction $\B\!\int\!$ and the constant embedding $\Delta$ define an equivalence of the localized categories
\[
\textstyle\B\!\int\colon \text{$\Br$-$\Cat^{\B}$}[w_{\B}^{-1}]\simeq
\text{$\Br$-$\Cat$}[w^{-1}]:\!\Delta
\]
and every object in $\Br$-$\Cat^{\B}$ is naturally $\B$-equivalent to a strictly commutative $\B$-category monoid. 
\end{theorem}  
Thus, working with braided monoidal categories is weakly equivalent to working with braided monoidal $\B$-categories and the latter category has the advantage that we may assume multiplications to be strictly commutative. 
This implies in particular that every braided monoidal small category is weakly equivalent to one of the form $\B\!\int\!A$ for a commutative $\B$-category monoid $A$.

Let $\Br$ be the categorical operad such that the category of $\Br$-algebras can be identified with $\Br$-$\Cat$ (see Section~\ref{subsec:operads} for details). For the analogous rectification in the category of spaces $\cS$ (which we interpret as the category of simplicial sets) we consider the operad $\nerve \Br$ in $\cS$ obtained by evaluating the nerve of $\Br$. This is an $E_2$ operad in the sense of being equivalent to the little 2-cubes operad and we may think of the category of algebras $\nerve\Br$-$\cS$ as the category of $E_2$ spaces. In order to rectify $E_2$ spaces to strictly commutative monoids we work in the diagram category of $\B$-spaces $\cS^{\B}$ equipped with the braided monoidal convolution product inherited from $\B$. There is an analogous category of $E_2$ $\B$-spaces $\nerve\Br$-$\cS^{\B}$. After localization with respect to the appropriate classes of $\B$-equivalences $w_{\B}$ in \mbox{$\nerve\Br$-$\cS^{\B}$} and weak equivalences $w$ in  
$\nerve\Br$-$\cS$, Proposition~\ref{prop:hocolim-B-equivalence} and Theorem~\ref{thm:Br-SB-rectification} combine to give the following result.

\begin{theorem}\label{thm:main-theorem-2}
The homotopy colimit $(-)_{h\B}$ and the constant embedding $\Delta$ define an equivalence of the localized categories
\[
(-)_{h\B}\colon \text{$\nerve\Br$-$\cS^{\B}$}[w_{\B}^{-1}]\simeq
\text{$\nerve\Br$-$\cS$}[w^{-1}] :\!\Delta
\]
and every object in $\nerve\Br$-$\cS^{\B}$ is naturally $\B$-equivalent to a strictly commutative $\B$-space monoid. 
\end{theorem}
This implies in particular that every double loop space is equivalent to an $E_2$ space of the form $A_{h\B}$ for a commutative $\B$-space monoid $A$. To give an example why this may be useful, notice that if $A$ is a commutative $\B$-space monoid, then the category $\cS^{\B}/A$ of $\B$-spaces over $A$ inherits the structure of a braided monoidal category. It is less obvious how to define such a structure for the corresponding category of spaces over an $E_2$ space.

The above rectification theorems have corresponding versions for symmetric monoidal categories and $E_{\infty}$ spaces that we spell out in Section~\ref{sec:I-section}. As an application of this we show how to rectify certain $E_{\infty}$ ring spectra to strictly commutative symmetric ring spectra. However, the braided monoidal setting is somewhat more subtle and is the main focus of this paper.

Our main tool for replacing braided monoidal structures by strictly commutative ones is a refinement of the usual strictification construction used to replace monoidal categories by strictly monoidal ones, see e.g.\ \cite[Section~1]{JS}. While it is well-known that this construction cannot be used to turn braided monoidal categories into categories with a strictly commutative multiplication, we shall see that it can be reinterpreted so as to take values in commutative $\B$-category monoids instead. This gives rise to the \emph{$\B$-category rectification functor} $\Phi$ introduced in Section~\ref{subsec:rectification}. In order to obtain an analogous rectification on the space level we apply the results of Fiedorowicz-Stelzer-Vogt~\cite{Fiedorowicz-V_simplicial,FSV} that show how to associate braided monoidal categories to $E_2$ spaces. Our rectification functor $\Phi$ then applies to these braided monoidal categories and we can apply the nerve functor level-wise to get back into the category of commutative $\B$-space monoids. 

It was pointed out by Stasheff and proved by Fiedorowicz~\cite{fied} and Berger \cite{Berger_double} that the classifying space of a braided monoidal small category becomes a double loop space after group completion. As an application of our techniques we show in Section~\ref{sec:classifying} how one can very simply define the double delooping: Given a braided monoidal category $\cA$, we apply the rectification functor $\Phi$ and the level-wise nerve to get a commutative $\B$-space monoid $\nerve\Phi(\cA)$. The basic fact (valid for any commutative monoid in a braided monoidal category whose unit is terminal) is now that the bar construction applied to $\nerve\Phi(\cA)$ is a simplicial monoid and hence can be iterated once to give a bisimplicial $\B$-space. Evaluating the homotopy colimit of this $\B$-space we get the double delooping. This construction in fact gives an alternative proof of Stasheff's result independent of the operadic recognition theorem for double loop spaces. 

Another ingredient of our work is a general procedure for constructing equivalences between localized categories that we detail in Appendix~\ref{app:localization}. This improves on previous work by 
Fiedorowicz-Stelzer-Vogt~\cite[Appendix~C]{FSV} and has subsequently been used by these authors in \cite{FSV2} to sharpen some of their earlier results.

\subsection{Organization}
We begin by introducing the category of braided injections in Section~\ref{sec:braided-injections} and establish the basic homotopy theory of $\B$-spaces in Section~\ref{sec:homotopy-B-spaces}. Then we switch to the categorical setting in 
Section~\ref{sec:B-categories} where we prove Theorem~\ref{thm:main-theorem-1}. In Section~\ref{sec:E2-spaces} we return to the analysis of $\B$-spaces and prove Theorem~\ref{thm:main-theorem-2}, whereas Section~\ref{sec:classifying} is dedicated to double deloopings of commutative $\B$-space monoids. Finally, we consider the symmetric monoidal version of the theory and relate this to the category of symmetric spectra in Section~\ref{sec:I-section}. The material on localizations of categories needed for the paper is collected in Appendix~\ref{app:localization}. 

\section{The category of braided injections}\label{sec:braided-injections}

We generalize the geometric definition of the braid groups by introducing the notion of a \emph{braided injection}. In this way we obtain a category $\B$ of braided injections such that the classical braid groups appear as the endomorphism monoids. 

In the following we write $I$ for the unit interval. Let $\mathbf{n}$ denote the ordered set $\{1,\ldots,n\}$ for $n\geq 1$. A braided injection $\alpha$ from 
$\mathbf{m}$ to $\mathbf{n}$, written $\alpha\colon \bld m\to\bld n$,  
is a homotopy class of $m$-tuples $(\alpha_1,\ldots,\alpha_m)$, where each $\alpha_i$ is a path 
$\alpha_i \co I \rightarrow \mathbb{R}^{2}$ starting in $(i,0)$ and ending in one of the points 
$(1,0),...,(n,0)$ with the requirement that $\alpha_i(t)\neq \alpha_j(t)$ for all $t$ in $I$, whenever $i\neq j$. 
Two $m$-tuples $(\alpha_1,\ldots,\alpha_m)$ and $(\beta_1,\ldots,\beta_m)$ are homotopic if 
there exists an $m$-tuple of homotopies $H_i \co I \times I \rightarrow \mathbb{R}^2$ 
from $\alpha_i$ to $\beta_i$, fixing endpoints, 
such that $H_i(s,t) \neq H_j(s,t)$ for all $(s,t)$ in $I\times I$ whenever $i\neq j$. 
The requirement that $H_i$ fixes endpoints ensures that a braided injection $\alpha$ from $\bld m$ to 
$\bld n$ defines an underlying injective function $\bar\alpha\colon \bld m\to \bld n$ by writing 
$\alpha_i(1)=(\bar\alpha(i),0)$. When visualising an injective braid, we think of the points $\alpha_i(t)$ for $i=1,\dots,m$ as a family of distinct points in $\mathbb{R}^2$ moving downwards from the initial position $(1,0)$, \dots, $(m,0)$, for $t=0$, to the final position $(\bar\alpha(1),0)$,\dots, 
$(\bar\alpha(m),0)$, for $t=1$.

\begin{figure}[ht]
\centering
\includegraphics[height=3cm]{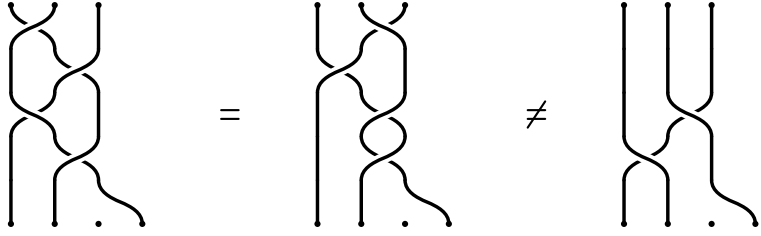} 
\caption{Braided injections with the same underlying injective map: $1 \mapsto 2$, $2 \mapsto 4$, $3 \mapsto 1$.}
\label{injective braids}
\end{figure}

We can compose two braided injections $\alpha\co \mathbf{m}\ra \mathbf{n}$ and 
$\beta\co\mathbf{n}\ra \mathbf{p}$ by choosing representatives 
$(\alpha_1,\ldots,\alpha_m)$ and $(\beta_1,\ldots,\beta_n)$, 
and set $\beta\circ\alpha$ to be the homotopy class of the  paths
$$(\beta_{\bar\alpha(1)}\cdot\alpha_1,...,\beta_{\bar\alpha(m)}\cdot\alpha_m)\text{.}$$ 
Here $\beta_{\bar\alpha(i)}\cdot\alpha_i$ denotes the usual composition of paths,
\[
\beta_{\bar\alpha(i)}\cdot\alpha_i(t)=
\begin{cases}
\alpha_i(2t), &\text{for $0\leq t\leq 1/2$},\\
\beta_{\bar\alpha(i)}(2t-1),& \text{for $1/2\leq t\leq1$}.
\end{cases}
\]
We let $\mathbf{0}$ denote the empty set and say that there is exactly one braided injection from 
$\mathbf{0}$ to $\mathbf{n}$ for $n\geq 0$.

\begin{definition}\label{def B}
The category $\B$ of braided injections has objects the finite sets $\bld n$ for $n\geq 0$ and morphisms the 
braided injections between these sets.
\end{definition}

Next we recall the definitions of some categories closely related to $\mathfrak{B}$.

\begin{definition}\label{def B S I M}
The categories $\mathcal{B}$, $\Sigma$, $\mathcal{I}$ and $\mathcal{M}$ all have as objects the finite sets 
$\mathbf{n}$ for $n\geq 0$. Here the \emph{braid category} 
$\mathcal{B}$ and the \emph{permutation category} $\Sigma$ have respectively the braid group 
$\mathcal{B}_n$ and the permutation group $\Sigma_n$ as the endomorphism set of $\mathbf{n}$, and no other morphisms. 
The morphisms in $\mathcal{I}$ and $\mathcal{M}$ are the injective functions and
the order preserving injective functions, respectively. 
\end{definition}

There is a canonical functor $\varPi$ from $\mathfrak{B}$ to $\mathcal{I}$ 
that takes a braided injection $\alpha\co\bld m \ra \bld n$ to the underlying injective 
function $\bar\alpha\co \bld m\ra\bld n$.
By definition, $\mathcal{B}$ is a subcategory of $\mathfrak{B}$ and 
$\Sigma$ is a subcategory of $\mathcal{I}$. Clearly $\varPi$ restricts to a functor from 
$\mathcal{B}$ to $\Sigma$, which we also denote by $\varPi$. 
The category $\mathcal{M}$ is a subcategory of $\mathcal{I}$ and there is a canonical embedding 
$\Upsilon\co\mathcal{M} \ra\mathfrak{B}$ with 
$\Upsilon(\mathbf{n})=\mathbf{n}$. For an injective order preserving function 
$\mu\co\mathbf{m}\ra\mathbf{n}$, let $\mu_i$  be the straight path from 
$(i,0)$ to $(\mu(i),0)$ for $1\leq i \leq m$.
Since $\mu$ is order preserving, $\mu_i(t)$ is different from $\mu_j(t)$ whenever $i\neq j$, and we can define $\Upsilon(\mu)$ as the braided injection represented by the tuple 
$(\mu_1,\ldots,\mu_m)$. 
 These functors fit into the following commutative diagram
\begin{equation}\label{diagram B B Sigma I M}
\xymatrix{
\mathcal{B} \ar@{}[r]|\subseteq \ar[d]_\varPi & \mathfrak{B} \ar[d]_\varPi \\
\Sigma \ar@{}[r]|\subseteq & \mathcal{I} & \mathcal{M}. \ar[lu]_\Upsilon
\ar@{}[l]|\supseteq 
} \end{equation}

The categories $\cB$, $\Sigma$, $\cI$ and $\cM$ are all monoidal categories with monoidal product 
$\sqcup$ given on objects by $\bld m\sqcup\bld n=\bld{m\!+\!n}$. In addition, $\cB$ is braided 
monoidal and $\Sigma$ and $\cI$ are symmetric monoidal. We will extend these monoidal structures 
to a braided monoidal structure on $\B$ such that all functors in the diagram are strict monoidal 
functors and functors between braided monoidal categories are braided strict monoidal functors. 
In order to do this, we will show that every morphism in $\B$ can be uniquely written in terms 
of a braid and a morphism in $\cM$.

\begin{lemma}\label{lemma unique decomposition}
Every braided injection $\alpha\colon \bld m\to\bld n$ can be written uniquely as a composition 
$\alpha=\Upsilon(\mu)\circ\zeta$ with $\mu$ in $\mathcal M(\bld m,\bld n)$ and $\zeta$ in the braid group $\mathcal B_m$.
\end{lemma}

\begin{proof}
Let $\mu\colon \bld m\to\bld n$ be the unique order preserving injective function whose image equals that of $\bar \alpha$, and let $\{j_1,\ldots,j_m\}$ be the permutation of the set $\mathbf{m}=\{1,\ldots,m\}$ determined by $\bar\alpha(i)=\mu(j_i)$ for $i=1,\dots,m$.  Choose representatives $(\mu_1,\ldots,\mu_m)$ and $(\alpha_1,\ldots,\alpha_m)$ for $\Upsilon(\mu)$ 
and $\alpha$ respectively. Let ${\mu}'_i$ be the reverse path of $\mu_i$ for $1\leq i\leq m$.
Since the path ${\mu}'_{j_i}$ starts in $(\mu(j_i),0)=\alpha_i(1)$ and ends in $(j_i,0)$, 
the homotopy class of the concatenated paths 
$({\mu}'_{j_1}\cdot\alpha_1,\ldots,{\mu}'_{j_m}\cdot\alpha_m)$ 
is a braid on $m$ strings and we define this to be $\zeta$. 
The composite $\Upsilon(\mu)\circ\zeta$ is represented by 
$(\mu_{j_1}\cdot{\mu}'_{j_1}\cdot\alpha_1,\ldots,\mu_{j_m}\cdot{\mu}'_{j_m}\cdot\alpha_m)$, 
which is clearly homotopic to  $(\alpha_1,\ldots,\alpha_m)$.
The morphism $\mu$ is uniquely determined by $\bar\alpha$ and we see from 
the construction that $\zeta$ is then also uniquely determined. 
\end{proof}

The above lemma implies that there is a canonical identification
\begin{equation}\label{eq:B-identification}
\B(\bld m,\bld n) \cong \mathcal M(\bld m,\bld n)\times \mathcal B_m.
\end{equation}
Now consider a pair $(\mu,\zeta)$ in $\mathcal M(\bld m,\bld n)\times \mathcal B_m$ and a pair $(\nu,\xi)$ in $\mathcal M(\bld n,\bld p)\times\mathcal B_n$. By Lemma \ref{lemma unique decomposition}  there exists a unique morphism $\xi_*(\mu)$ in $\mathcal M(\bld m,\bld n)$ and a unique braid $\mu^*(\xi)$ in $\mathcal B_m$ such that the diagram
$$\xymatrix@-1.pt{
\mathbf{m} \ar[rr]^{\Upsilon(\mu)} \ar[d]_{\mu^\ast(\xi)} && \mathbf{n} \ar[d]^\xi \\
\mathbf{m} \ar[rr]^{\Upsilon(\xi_\ast(\mu))} && \mathbf{n}
}$$
commutes in $\B$. Hence we see that composition in $\B$ translates into the formula
\[
(\nu,\xi)\circ(\mu,\zeta)=(\nu\circ\xi_*(\mu),\mu^*(\xi)\circ\zeta)
\]
under the identification in \eqref{eq:B-identification}.

In order to define functors out of the categories considered in Definition~\ref{def B S I M}, it is sometimes convenient to have these categories expressed in terms of generators and relations. Consider first the case of $\mathcal M$ and write $\partial_n^i\colon \bld n\to \bld{n\sqcup 1}$ for the morphism that misses the element $i$ in $\{1,\dots,n+1\}$. It is well known that 
 $\cM$ is generated by the morphisms  $\partial^i_n$ subject to the relations 
 \begin{equation*}
  \partial^i_{n+1}\partial^j_n = \partial^{j+1}_{n+1}\partial^{i}_n \quad\text{ for } i\leq j. 
 \end{equation*}

Now consider the category $\B$ and let $\zeta_n^1,\dots,\zeta_n^{n-1}$ be the standard generators for the braid group $\mathcal B_n$, see e.g. \cite[Theorem~1.8]{bir}.
\begin{figure}[ht]
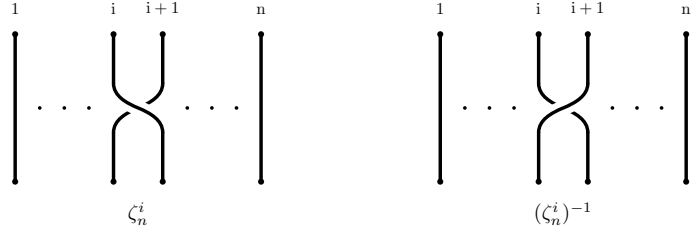

\centering \includegraphics[height=3cm]{zeta.mps} 
 \hspace{2cm}
 \includegraphics[height=3cm]{zeta-inverse.mps}
 \caption{The generator $\zeta_n^i$ and its inverse \label{fig generator zeta}}
\end{figure}

We also write $\partial_n^i\colon \bld n\to \bld{n\sqcup 1}$ for the braided injections obtained by applying the functor $\Upsilon$ to the corresponding morphisms in $\mathcal M$. 

\begin{lemma}\label{lemma generating morphisms for B}
The category $\mathfrak{B}$ is generated by the morphisms $\zeta^i_n\co \mathbf{n}\ra \mathbf{n}$ 
for $n\geq 2$ and $1\leq i \leq n-1$, and the morphisms $\partial^i_n\co \mathbf{n}\ra\mathbf{n\sqcup1}$ for 
$n\geq 0$ and $1\leq i \leq n\!+\!1$, subject to the relations 
\begin{align*}
&\!\begin{aligned}
&\zeta^i_n\zeta^j_n=\zeta^j_n\zeta^i_n &&\text{for } |i-j|\geq2 \\
&\zeta^i_n\zeta^{i+1}_n\zeta^i_n=\zeta^{i+1}_n\zeta^i_n\zeta^{i+1}_n &&\text{for } 1\leq i\leq n-2 \\
&\partial^i_{n+1}\partial^j_n = \partial^{j+1}_{n+1}\partial^{i}_n &&\text{for } i\leq j 
\end{aligned}\\
&\zeta^i_{n+1}\partial^j_n = 
\begin{cases} \partial^j_n\zeta^{i-1}_n & \text{for } j<i \\ 
\partial^{j+1}_n & \text{for } j=i \\
\partial^{j-1}_n & \text{for } j=i+1 \\
\partial^j_n\zeta^i_n & \text{for } j>i+1.
\end{cases}
\end{align*}
\end{lemma}

\begin{proof}
The identification $\mathfrak{B}(\mathbf{m},\mathbf{n}) \cong 
\mathcal{M}(\mathbf{m},\mathbf{n})\times \mathcal{B}_m$ 
makes it clear that any morphism can be written in terms of the generators. The two first 
relations are the relations for the braid groups (see e.g.,\ \cite[Theorem~1.8]{bir}), the next are the relations in $\mathcal{M}$, so that leaves the relations between the $\partial^i_n$'s and $\zeta^i_n$'s. 
It is easy to see that these relations hold in $\B$ and that they can be used to decompose 
any product of the $\partial^i_n$'s and the $\zeta^i_n$'s into the form $\Upsilon(\mu)\circ \zeta$ for a braid $\zeta$ and a morphism $\mu$ in $\mathcal M$. Since such a decomposition is 
unique, the relations are also sufficient. 
\end{proof}

Finally, we consider the category $\mathcal I$ and write $\sigma_n^i\colon \bld n\to\bld n$ for the image of 
$\zeta_n^i$ under the projection $\varPi\colon \B\to\mathcal I$. 
We  obtain a presentation of $\cI$ from the presentation of $\B$ by imposing the relation $\sigma_n^i\sigma_n^i=\id_n$, just as the symmetric group $\Sigma_n$ is obtained from $\mathcal B_n$.

We use the above to define a strict monoidal structure on $\B$ with unit $\bld 0$. 
Just as for the monoidal categories considered in Diagram \eqref{diagram B B Sigma I M}, the monoidal product $\bld m\sqcup \bld n$ 
of two objects $\bld m$ and $\bld n$ in $\B$ is $\bld{m\!+\!n}$. 
The decomposition of a braided injection given in \eqref{eq:B-identification} lets us define the monoidal product $(\mu,\zeta)\sqcup(\nu,\xi)$ of two morphisms $(\mu,\zeta)$ and $(\nu,\xi)$ in $\B$ as $(\mu\sqcup\nu,\zeta\sqcup\xi)$ using the monoidal structures on $\cM$ and $\cB$, for an illustration of this see Figure \ref{fig monoidal product}. 

\begin{figure}[ht]
\centering
\includegraphics[height=1.5cm]{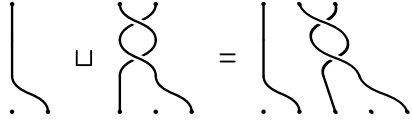} 
\caption{The monoidal product of two braided injections. \label{fig monoidal product}}
\end{figure}

It is well known that the subcategory $\cB$ is braided with braiding $\chi_{\bld m, \bld n}\co \bld m\sqcup \bld n \ra \bld n\sqcup \bld m$ moving the  first $m$ strings over the  last $n$ strings while keeping the order among the $m$ strings and the $n$ strings respectively. This family of isomorphisms is in fact also a braiding on $\B$. The hexagonal axioms for a braiding only involve morphisms in $\cB$ 
so it remains to check that $\chi_{\bld m, \bld n}$ is natural with respect to the generators $\partial^i_k$. This is quite clear geometrically (see Figure \ref{fig naturality chi} for an illustration) and can be checked algebraically by writing $\chi_{\bld m, \bld n}$ in terms of the generators.

\begin{figure}[ht]
\centering
\includegraphics[height=3cm]{naturality-chi.mps} 
\caption{The equality $(\partial^2_3\sqcup\id _\bld 2)\circ \chi_{\bld 2, \bld 3}=\chi_{\bld 2, \bld 4}\circ(\id_\bld2\sqcup \partial^2_3)$. \label{fig naturality chi}}
\end{figure}

\section{The homotopy theory of $\mathfrak{B}$-spaces}\label{sec:homotopy-B-spaces}

In this section we introduce $\B$-spaces as functors from $\B$ to the category of spaces and  
equip  the category of $\B$-spaces with a braided monoidal model structure. We assume some familiarity 
with the basic theory of cofibrantly generated model categories as presented in \cite[Section~2.1]{hov} and 
\cite[Section~11]{hirsch}.

\subsection{The category of $\mathfrak{B}$-spaces}
A $\mathfrak{B}$-space is a functor $X\co \mathfrak{B}\rightarrow\mathcal{S}$, where $\mathcal{S}$ 
is the category of simplicial sets. We call a natural transformation between two such functors a 
morphism between the two $\B$-spaces and write $\SB$ for the category of $\B$-spaces so defined.

The category $\mathcal{S}^\mathfrak{B}$ inherits much structure from $\cS$. All small limits and colimits 
exists and are constructed level-wise. Furthermore, $\SB$ is enriched, tensored and cotensored over $\mathcal{S}$. 
For a $\mathfrak{B}$-space $X$ and a simplicial set $K$, the tensor $X\times K$ and cotensor 
$X^K$ are the $\mathfrak{B}$-spaces given in level $\mathbf{n}$ by 
$$(X\times K)(\mathbf{n})=X(\mathbf{n})\times K \quad\text{ and } \quad
X^K(\mathbf{n})=\mathrm{Map}_\mathcal{S}(K,X(\mathbf{n})),$$
where $\mathrm{Map}_\mathcal{S}$ is the standard simplicial function complex.
The simplicial set of maps from $X$ to $Y$ is the end 
$$\mathrm{Map}_\SB(X,Y)= 
\int_{\mathbf{n}\in\mathfrak{B}}\mathrm{Map}_\mathcal{S}\big(X(\mathbf{n}),Y(\mathbf{n})\big).$$

\begin{lemma}
 The category of $\mathfrak{B}$-spaces is a bicomplete simplicial 
 category with the above defined structure. \qed 
\end{lemma}

\subsection{The $\mathfrak{B}$-model structure on $\mathcal{S}^\mathfrak{B}$}
We will use the free $\B$-space functors $F_\mathbf n\co \cS \ra \SB$ given by $F_{\mathbf n}(K)=\B(\mathbf n,-)\times K$ to transport the usual model structure on simplicial sets to $\SB$. The functor $F_\mathbf n$ is left adjoint to the evaluation functor $\mathrm{Ev}_\mathbf n$ taking a $\B$-space $X$ to the simplicial set $X(\mathbf n)$. Note that since $\mathbf 0$ is initial in $\B$, the functor $F_\mathbf 0$ takes a simplicial set to a constant $\B$-space. We often use the notation $\Delta$ for $F_\mathbf 0$. 

It is a standard fact, see for instance \cite[Theorem~11.6.1]{hirsch}, that $\SB$ has a level model structure where a morphism is a weak equivalence (or respectively a fibration) if it is a weak equivalence (or respectively a fibration) of simplicial sets when evaluated at each level $\mathbf n$. This model structure is cofibrantly generated with generating cofibrations 
\begin{equation*}
 I=\{F_{\mathbf n}(i)\,\,|\,\, \mathbf n \in \B,\, i\co \partial\Delta^k\ra\Delta^k \text{ for } 0\leq k\}
\end{equation*}
and generating acyclic cofibrations
\begin{equation*}
 J=\{F_{\mathbf n}(j)\,\,|\,\, \mathbf n \in \B, \, j\co \Lambda^k_l\ra\Delta^k \text{ for } k>0 \text{ and }0\leq l\leq k\}
\end{equation*}
where $i$ and $j$ denote the inclusion of the boundary of $\Delta^k$ and the $l$th horn of $\Delta^k$ in 
$\Delta^k$ respectively. The cofibrations in the level model structure have a concrete description using latching maps. The $n$th latching space of a $\B$-space $X$ is defined as
$$L_{\bld n}(X)=\colim_{(\bld m \ra\bld n)\in\partial(\B\downarrow\bld n)}X(\bld m),$$
where $\partial(\B\!\downarrow\!\bld n)$ is the full subcategory of the 
comma category $(\B\!\downarrow\!\bld n)$ with objects the non-isomorphisms.
For a map of $\B$-spaces $f\co X\ra Y$, the $n$th latching map is the $\cB_n$-equivariant map
$L_{\bld n}f\co L_{\bld n}(Y)\amalg_{L_{\bld n}(X)} X(\bld n) \ra Y(\bld n)$. A map $f\co X\ra Y$ is a cofibration if for every $n\geq0$, the $n$th latching map $L_{\bld n}f$
 is a cofibration of simplicial sets such that the $\cB_n$-action on the complement of the image is free.  We refer to such cofibrations as $\B$-cofibrations. 

 The level model structure is primarily used as a convenient first step in equipping $\cS^{\B}$ with a model structure making it Quillen equivalent to $\cS$. In such a model structure we need a wider class of weak equivalences. Recall that the Bousfield-Kan construction of the homotopy colimit of a functor 
$X$ from a small category $\cC$ to $\cS$  is the simplicial set $\hocolim_\cC X$ with $k$-simplices  
\begin{equation}\label{hocolim}
\coprod_{\bld m_0\leftarrow\cdots\leftarrow \bld m_k}X(\bld m_k)_k
\end{equation}
for morphisms $\bld m_0\leftarrow \bld m_1,\dots,\bld m_{k-1}\leftarrow \bld m_k$ in $\cC$, cf.\  \cite[Section~XII.5.1]{BK}. When the functor $X$ is a $\B$-space we will often denote its homotopy colimit by $X\hB$.

\begin{definition}\label{def B-equivalence}
 A morphisms $X\ra Y$ of $\B$-spaces is a $\B$-equivalence if the induced map 
 $X\hB \ra Y\hB$  is a weak equivalence of simplicial sets. 
\end{definition}

We say that a morphism $X\ra Y$ of $\B$-spaces is a $\B$-fibration if $X(\bld n)\ra Y(\bld n)$ is a fibration of simplicial sets for every $\bld n\in \B$ 
and if the square
\begin{equation}\label{diagram homotpy cartesian}
\xymatrix{
X(\mathbf{m}) \ar[r]^{X(\alpha)} \ar[d] & X(\mathbf{n}) \ar[d] \\
Y(\mathbf{m}) \ar[r]^{Y(\alpha)} & Y(\mathbf{n}) 
}\end{equation}
is homotopy cartesian for every braided injection 
$\alpha \co \mathbf{m}\rightarrow \mathbf{n}$.

 In order to make the $\B$-equivalences and $\B$-fibrations part of a cofibrantly generated model structure we have to add more generating acyclic cofibrations compared to the level model structure. We follow the approach taken for diagram spectra in \cite[Section~3.4]{HSS} and \cite[Section~9]{MMSS} and for diagram spaces in \cite[Section~6.11]{Sagave-Schlichtkrull}: Each braided injection $\alpha\co\mathbf m\ra \mathbf n$ gives rise to a map of $\B$-spaces $\alpha^\ast\co F_{\mathbf{n}}(\ast) \ra F_{\mathbf{m}}(\ast)$. The latter map factors through the mapping cylinder $M(\alpha^\ast)$ as $\alpha^\ast=r_\alpha j_\alpha$, where $j_\alpha$ is a cofibration in the level model structure and $r_\alpha$ is a simplicial homotopy equivalence.
We now set 
$$\bar{J}= \{j_\alpha\Box i\,\,|\,\, \alpha\co\bld m\ra \bld n \in \B,\, i\co \partial\Delta^k\ra\Delta^k \text{ for } 0\leq k\},$$
where $\Box$ denotes the pushout-product, see e.g.\ \cite[Definition~4.2.1]{hov}.

\begin{proposition}\label{prop B-model structure}
There is a model structure on $\SB$, the $\B$-model structure, with weak equivalences the $\B$-equivalences, fibrations the $\B$-fibrations and cofibrations the $\B$-cofibrations. This model structure is simplicial and cofibrantly generated 
where $I_\B=I$ is the set of generating cofibrations and $J_\B=J\cup \bar{J}$ is the set of generating acyclic cofibrations. 
\end{proposition}

\begin{proof}
The proof is similar to the proofs of Propositions~6.16 and 6.19 in \cite{Sagave-Schlichtkrull}. (We refer the reader to Remark~\ref{rem:sagave-schlichtkrull-remark} for a summary of the extent to which the results for symmetric monoidal diagram categories established in \cite{Sagave-Schlichtkrull} carries over to the present setting.) 
\end{proof}

As promised this model structure makes $\B$-spaces Quillen equivalent to simplicial sets.

\begin{proposition}
 The adjunction 
 $\colim_\B\co\SB \rightleftarrows \cS \thinspace\colon \Delta$
 is a Quillen equivalence.
\end{proposition}

\begin{proof}
 The category  $\B$ has an initial object so $\nerve \B$ is a contractible simplicial set. 
  Arguing as in the proof of Proposition~6.23 in \cite{Sagave-Schlichtkrull} yields the result. 
\end{proof}

\begin{example}\label{ex:X-bullet-example-1}
In general an $\cI$-space $Z\colon \cI\to \cS$ pulls back to a $\B$-space $\Pi^*Z$ via the functor $\Pi\colon\B\to \cI$ from Section~\ref{sec:braided-injections}. Consider in particular a based space $X$ with base point $*$ and the $\cI$-space $X^{\bullet}\colon \cI\to\cS$ such that $X^{\bullet}(\bld n)=X^n$. A morphism $\alpha\colon\bld m\to\bld n$ in $\cI$ acts on an element $\bld x=(x_1,\dots,x_m)$ by 
\[
\alpha_*(\bld x)=(x_{\alpha^{-1}(1)},\dots,
x_{\alpha^{-1}(n)}), 
\]
where $x_{\alpha^{-1}(j)}=x_i$ if $\alpha(i)=j$ and $x_{\alpha^{-1}(j)}=*$ if $j$ is not in the image of $\alpha$.
It is proved in \cite{SHom} that if $X$ is connected, then the geometric realization $|X^{\bullet}_{h\cI}|$ is equivalent to the infinite loop space $\Omega^{\infty}\Sigma^{\infty}(|X|)$. In contrast to this we shall prove in Example~\ref{ex:X-bullet-example-2} that $|(\Pi^*X^{\bullet})_{h\B}|$ is equivalent to $\Omega^2\Sigma^2(|X|)$ for connected $X$.
\end{example}

\subsection{The flat $\mathfrak{B}$-model structure on $\mathcal{S}^\mathfrak{B}$}

We will now consider another structure on $\B$-spaces, the flat $\B$-model structure, which takes into account that each level of a $\B$-space has a left action of a braid group. The weak equivalences are again the 
$\B$-equivalences, but the flat $\B$-model structure has more cofibrant objects than the $\B$-model structure. In some places, in particular in Section~\ref{sec:classifying}, we get more general results by considering these ``flat'' objects instead of only the $\B$-cofibrant objects. The flat $\B$-model structure is constructed similarly to the $\B$-model structure, but the starting point is Shipley's mixed model structure on the category  $\cB_n$-$\cS$ of  simplicial sets with left $\cB_n$-action, see \cite[Proposition~1.3]{Shipley}. Shipley only considers finite groups, but the construction applies equally well to discrete groups in general if one allows all subgroups to be considered. An equivariant map is a weak equivalence (or respectively a cofibration) in the  mixed model structure  if the underlying map of simplicial sets is. 
Recall that given a group $H$ and an $H$-space $K$, the space of homotopy fixed points $K^{hH}$ 
is the homotopy limit of $K$
viewed as a diagram over the one-object category $H$. An equivariant map $K\ra L$ 
is a fibration in the  mixed model structure if the induced maps $K^H\ra L^H$ of fixed points are fibrations and the diagrams 
$$\xymatrix{
K^H \ar[r] \ar[d] & K^{hH} \ar[d] \\
L^H \ar[r]& L^{hH} 
}$$
are homotopy cartesian for all subgroups $H$ of $\cB_n$. This model structure is cofibrantly generated, see the proof of \cite[Proposition~1.3]{Shipley} for a description of the generating (acyclic) cofibrations.

The forgetful functor $\mathrm{Ev}_{\mathbf n}\co \SB \ra \cB_n\text{-}\cS$ evaluating a $\B$-space $X$ at the $n$th level has a right adjoint $G_{\mathbf n}$ given by $G_{\mathbf n}(K)=\B(\mathbf n,-)\times_{\cB_n} K$. We proceed as in the previous subsection and get a new level model structure on 
$\SB$ where  a morphism is a weak equivalence (or respectively a fibration) if it is a weak equivalence (or respectively a fibration) in the mixed model structure on $\cB_n$-$\cS$ when evaluated at each level $\mathbf n$. This model structure is cofibrantly generated with generating (acyclic) cofibrations $I_\mathrm{m}$ ($J_\mathrm{m}$) obtained by applying $G_\mathbf n$  to the generating (acyclic) cofibrations for the mixed model structure on $\cB_n$-$\cS$ for all $\mathbf n$ in $\B$. 
A morphisms $f\co X\ra Y$ is a cofibration in this level model structure if for every $n\geq0$, the $n$th latching map $L_{\bld n}f$
 is a cofibration of simplicial sets. We refer to such cofibrations as flat $\B$-cofibrations. 
A morphism $X\ra Y$ of $\B$-spaces is said to be a flat $\B$-fibration if $X(\bld n)\ra Y(\bld n)$ is a fibration in  the mixed model structure on $\cB_n$-$\cS$ for every $\bld n$ in $ \B$ 
and if the square \eqref{diagram homotpy cartesian}
is homotopy cartesian for every braided injection 
$\alpha \co \mathbf{m}\rightarrow \mathbf{n}$.

\begin{proposition}\label{prop flat model structure}
There is a model structure on $\SB$, the flat $\B$-model structure, with weak equivalences the $\B$-equivalences, fibrations the flat $\B$-fibrations and cofibrations the flat $\B$-cofibrations. This model structure is simplicial and cofibrantly generated 
where $I_{\mathrm{flat}}=I_\mathrm{m}$ is the set of generating cofibrations and $J_{\mathrm{flat}}=J_\mathrm{m}\cup \bar{J}$ is the set of generating acyclic cofibrations. 
\end{proposition}

\begin{proof}
The proof is similar to the proofs of Propositions~6.16 and 6.19 in \cite{Sagave-Schlichtkrull}. 
\end{proof}

We will refer to the flat $\B$-cofibrant objects simply as flat objects. These will play an important role also when we are considering the $\B$-model structure.
The next result gives a criterion for an object to be flat which is easier to check than the one given above.

\begin{proposition}\label{prop flat B-space}
A $\B$-space $X$ is flat if and only if each
morphism $\bld m \ra \bld n$ induces a cofibration $X(\bld m) \ra X(\bld n)$ 
and for each diagram of the
following form (with maps induced by the evident order preserving morphisms)
\begin{equation}\label{diagram condition for flatness}
 \xymatrix{
 X(\bld m) \ar[r] \ar[d] & X(\bld{m\sqcup n}) \ar[d] \\
 X(\bld{l\sqcup m}) \ar[r] & X(\bld{l\sqcup m\sqcup n})}
\end{equation}
the intersection of the images of $ X(\bld{l\sqcup m})$ and $X(\bld{m\sqcup n})$ 
in $ X(\bld{l\sqcup m\sqcup n})$ equals the image of $X(\bld m)$.
\end{proposition}

\begin{proof}
 Recall from Definition~\ref{def B S I M} the canonical embedding 
 $\Upsilon\co\cM\ra\B$, where $\cM$ is the category with the same objects as $\B$ 
 and injective order preserving functions as morphisms. This induces an embedding $(\cM\!\downarrow\!\bld n) \ra (\B\!\downarrow\!\bld n)$
whose image is a skeletal subcategory by Lemma~\ref{lemma unique decomposition}. Identifying 
$(\cM\!\downarrow\!\bld n)$ with the poset category  of subsets of $\bld n$, we see that a $\B$-space gives rise to an $\bld n$-cubical diagram for all $\bld n$. Furthermore, it follows from the definitions that a map of $\B$-spaces is a flat $\B$-cofibration if and only if the induced maps of cubical diagrams are cofibrations in the usual sense. Given this, the proof proceeds along the same lines as the proof of the analogous result for $\cI$-spaces, see \cite[Proposition~3.11]{Sagave-Schlichtkrull}.
\end{proof}

\subsection{The braided monoidal structure on $\mathcal{S}^\mathfrak{B}$}\label{subsec:braided-monoidal-B-spaces}
Any category of diagrams in $\mathcal{S}$ indexed by a  braided monoidal small category inherits a 
braided monoidal convolution product from the indexing category. We proceed to explain how this works in the case of $\cS^{\B}$. Given $\B$-spaces $X$ and $Y$, we define the $\B$-space $X\boxtimes Y$ to be the left Kan extension of the ($\B\times \B$)-space
\[
\B\times \B \xr{X\times Y} \cS\times\cS \xr{\times} \cS
\]
along the monoidal structure map $\sqcup\colon \B\times\B\to \B$. Thus, the data specifying a map of 
$\B$-spaces $X\boxtimes Y\to Z$ is equivalent to the data giving a map of ($\B\times \B$)-spaces $X(\bld m)\times Y(\bld n)\to Z(\bld m\sqcup\bld n)$. We also have the level-wise description 
\[
X\boxtimes Y(\bld n)=\colim_{\bld n_1\sqcup\bld n_2\to\bld n} X(\bld n_1)\times Y(\bld n_2)
\]
where the colimit is taken over the comma category $(\sqcup\!\downarrow\!\bld n)$ associated to the monoidal product $\sqcup\co \B\times\B\ra\B$.
 The monoidal unit for the $\boxtimes$-product is the terminal $\B$-space 
$U^{\B}=\B(\bld 0,-)$. Using that $\cS$ is cartesian closed one easily defines the coherence isomorphisms for associativity and unity required to make $\cS^{\B}$ a monoidal category. We specify a braiding
$\br\colon X\boxtimes Y\to Y\boxtimes X$ on $\cS^{\B}$ by requiring that the diagram of 
($\B\times \B$)-spaces
\begin{equation}\label{eq:B-spaces-braiding}
\xymatrix{
X(\bld m)\times Y(\bld n) \ar[rrr]^{\mathrm{twist}} \ar[d] &&&Y(\bld n)\times X(\bld m) \ar[d]\\
X\boxtimes Y(\bld m\sqcup \bld n) \ar[r]^{\br(\bld m\sqcup \bld n)} &
Y\boxtimes X(\bld m\sqcup \bld n) \ar[rr]^{Y\boxtimes X(\chi_{\bld m,\bld n})} &&
Y\boxtimes X(\bld n\sqcup \bld m)
}
\end{equation}
be commutative. The following proposition can either be checked by hand or deduced from the general theory in \cite{Day}. 
\begin{proposition}\label{prop:B-spaces-braided-monoidal}
The category $\cS^{\B}$ equipped with the $\boxtimes$-product, the unit $U^{\B}$, and the braiding $\br$ is a braided monoidal category. \qed
\end{proposition}

We use the term \emph{$\B$-space monoid} for a monoid in $\SB$. By the universal property of the 
$\boxtimes$-product, the data needed to specify the unit $u\colon U^{\B}\to A$ and the multiplication 
$\mu\colon A\boxtimes A\to A$ on a $\B$-space monoid $A$ amounts to a zero simplex $u$ in $A(\bld 0)$ and a map of ($\B\times \B$)-spaces $\mu\colon A(\bld m)\times A(\bld n)\to A(\bld m\sqcup \bld n)$ satisfying the usual associativity and unitality conditions. By the definition of the braiding, $A$ is commutative (that is, 
$\mu\circ \br=\mu$) if and only if the diagram of ($\B\times\B$)-spaces 
\begin{equation}\label{eq:commutativity-condition}
\xymatrix{
A(\bld m)\times A(\bld n) \ar[r]^-{\mu} \ar[d]_{\mathrm{twist}}& A(\bld m\sqcup\bld n) \ar[d]^{A(\chi_{\bld m,\bld n})} \\
A(\bld n)\times A(\bld m) \ar[r]^-{\mu} & A(\bld n\sqcup\bld m)
}
\end{equation}
is commutative.

 Recall that given maps $f_1\co X_1\ra Y_1$ and $f_2\co X_2\ra Y_2$ of $\B$-spaces, 
the pushout-product  is the induced map 
$$f_1\Box f_2\co(X_1\boxtimes Y_2)\amalg_{(X_1\boxtimes X_2)}
(Y_1\boxtimes X_2) \ra Y_1\boxtimes Y_2.$$
Following \cite[Definition~4.2.6]{hov} we say that a model structure on $\SB$ is a monoidal model structure if  given any two cofibrations $f_1$ and $f_2$, the pushout-product $f_1\Box f_2$ 
 is a cofibration that is in addition acyclic if $f_1$ or $f_2$ is.

\begin{lemma}\label{lemma pushout-product axiom}
Both the $\B$-model structure and the flat $\B$-model structure 
are monoidal model structures.
\end{lemma}

\begin{proof}
We give the proof for the $\B$-model structure, the proof for the flat case is similar.
By Lemma~3.5 in \cite{SSAlgebrasModules} it suffices to verify the condition for the generating (acyclic) cofibrations. For two generating cofibrations $i,i'$ in $\cS$ it is easy to check that 
$F_\mathbf{m}(i)\Box F_\mathbf{n}(i')$ is isomorphic to $F_{\mathbf{m}\sqcup\mathbf{n}}(i\Box i')$. 
This uses that $F_\mathbf{m}(K)\boxtimes F_\mathbf{n}(L)$ is naturally isomorphic
to $F_{\mathbf{m}\sqcup\mathbf{n}}(K\times L)$ for two simplicial sets $K$ and $L$  and also
that $F_{\mathbf{m}\sqcup\mathbf{n}}$ is a left adjoint and hence commutes with colimits.
Simplicial sets is a monoidal model category, therefore $i\Box i'$ is a cofibration and then so 
is $F_{\mathbf{m}\sqcup\mathbf{n}}(i\Box i')$, since $F_{\mathbf{m}\sqcup\mathbf{n}}$ preserves cofibrations.
Similarly $F_\mathbf{m}(i)\Box F_\mathbf{n}(j)$ is an acyclic cofibration if 
$j$ is a generating acyclic cofibration in $\mathcal{S}$.

Now let $\alpha\co \bld m\ra \bld m'$ be a morphism in $\B$. We  check that 
$(j_\alpha\Box i)\Box F_\mathbf{n}(i')$ is an acyclic cofibration for 
$i\co \partial\Delta^k\ra\Delta^k$ and $i'\co \partial\Delta^l\ra\Delta^l$ generating cofibrations in $\cS$.
Using that $j_{\alpha}\Box i\cong j_{\alpha}\Box F_{\bld 0}(i)$, we get the identifications
\[
(j_\alpha\Box i)\Box F_\mathbf{n}(i')\cong j_\alpha\Box(F_{\bld 0}(i)\Box F_\mathbf{n}(i'))\cong
j_{\alpha}\Box F_{\bld n}(i\Box i')\cong j_{\alpha}\boxtimes F_{\bld n}(*)\times (i\Box i').
\]
Since $j_{\alpha}$ is a cofibration by construction, it follows from the first part of the lemma and the fact that 
the $\B$-model structure is simplicial, that this is a cofibration. For the same reason it therefore suffices to show that $j_{\alpha}\boxtimes F_{\bld n}(*)$ is a $\B$-equivalence.  
For this we apply the two out of three property for $\B$-equivalences to the diagram
$$\xymatrix{
F_\mathbf{m}(\ast)\boxtimes F_\mathbf{n}(*) \ar[d]^{\cong} 
\ar[rr]^{j_\alpha\boxtimes \mathrm{id}_{F_\mathbf{n}(*)}} 
&& M(\alpha^\ast)\boxtimes F_\mathbf{n}(*)
\ar[rr]^{r_\alpha\boxtimes \mathrm{id}_{F_\mathbf{n}(*)}} 
&& F_{\mathbf{m}'}(\ast)\boxtimes F_\mathbf{n}(*) \ar[d]^{\cong}\\
F_{\mathbf{m}\sqcup\mathbf{n}}(\ast) 
\ar[rrrr]^{(\alpha\sqcup\mathrm{id}_\mathbf{n})^\ast}_\sim
&&&& F_{\mathbf{m}'\sqcup\mathbf{n}}(\ast).
}$$
The vertical maps are isomorphisms and the lower horizontal map is a $\B$-equivalence since both 
$F_{\bld m\sqcup\bld n}(\ast)\hB$ and $F_{\bld m'\sqcup\bld n}(\ast)\hB$ are contractible. 
Furthermore, $r_\alpha\boxtimes\id_{F_\bld n(*)}$ is a simplicial homotopy equivalence 
since $r_\alpha$ is a simplicial homotopy equivalence and $-\boxtimes \mathrm{id}_{F_\mathbf{n}(*)}$ 
preserves simplicial homotopy equivalences. This completes the proof. 
\end{proof}

\begin{remark}\label{rem:sagave-schlichtkrull-remark}
 In \cite{Sagave-Schlichtkrull} a projective model structure is defined for a general diagram category 
 $\cS^\cK$ indexed by a small symmetric monoidal category $\cK$ that is well-structured as per 
 Definition~5.5 in \cite{Sagave-Schlichtkrull}. Similarly a flat model structure is defined for 
 $\cS^\cK$ if in addition $\cK$ together with its subcategory of automorphisms form a well-structured 
 relative index category as per Definition~5.2 in \cite{Sagave-Schlichtkrull}. These definitions can be 
 canonically extended to allow braided monoidal categories as index categories such that similar 
 model structures exist. This will not make $\B$ a well-structured index category because the comma category 
 $(\bld k\sqcup -\!\downarrow\!\bld l)$ will in general not have a terminal object for $\bld k$ 
 and $\bld l$ in $\B$. This property is however not used to establish the model structures, so 
 Proposition~\ref{prop B-model structure} and Proposition~\ref{prop flat B-space} are proved as 
 the similar results in \cite{Sagave-Schlichtkrull}. 
 But the proofs of results concerning how the monoidal structure interacts with the model structures do use the mentioned property. Above we have shown that the model structures we consider are monoidal model structures by an alternative argument. It is not clear if the arguments in \cite{Sagave-Schlichtkrull} can be generalized to define model structures on monoids and commutative monoids in $\B$-spaces.
\end{remark}

Let $X$ and $Y$ be $\B$-spaces and consider the natural transformation
$$\nu_{X,Y}\co X\hB \times Y\hB \xrightarrow{\cong}(X\times Y)_{h(\B\times\B)}\ra 
((-\sqcup-)^\ast(X\boxtimes Y))_{h(\B\times\B)}\ra (X\boxtimes Y)\hB$$
where the second map is induced by the universal natural transformation of $\B\times\B$ diagrams
$X(\bld m) \times  Y(\bld n) \ra (X\boxtimes Y )(\bld{m\sqcup n})$. 
These maps gives rise to a monoidal
structure on the functor $(-)\hB$, c.f.\ \cite[Proposition~4.17]{SThom}.

\begin{lemma}\label{lemma X and Y flat implies weak equivalnece between times and boxtimes}
 If both $X$ and $Y$ are flat, then 
 $\nu_{X,Y}\co X\hB \times Y\hB \ra (X\boxtimes Y)\hB$ is a weak equivalence. 
\end{lemma}

\begin{proof}
The fact that the flat $\B$-model structure is monoidal combined with Ken Brown's Lemma implies that the functor $X\boxtimes (-)$ takes $\B$-equivalences between flat $\B$-spaces to $\B$-equivalences since $X$ is itself flat. Therefore we can take a cofibrant replacement of $Y$ in the $\B$-model structure and it will suffice to prove the result when $Y$ is $\B$-cofibrant. Applying a symmetric argument we reduce to the case where both $X$ and $Y$ are $\B$-cofibrant, which in turn implies that also $X\boxtimes Y$ is $\B$-cofibrant. By Proposition~18.9.4 in \cite{hirsch}
the canonical map $\hocolim_\B Z \ra \colim_B Z$ is a weak equivalence for any $\B$-cofibrant $\B$-space $Z$. The claim now follows since the colimit functor is strong symmetric monoidal.
\end{proof}

\section{$\mathfrak{B}$-categories and braided monoidal structures}\label{sec:B-categories}
In this section we introduce the notion of a $\B$-category and equip the category of such with a braided monoidal structure. We then relate the braided (strict) monoidal objects in this setting to braided (strict) monoidal categories in the usual sense.  Finally we introduce the $\B$-category rectification functor and use this to show that any braided monoidal structure can be rectified to a strictly commutative structure up to $\B$-equivalence. 

\subsection{$\mathfrak{B}$-categories and the Grothendieck construction}\label{subsec:B-categories-Grothendieck}
Let $\Cat$ denote the category of small categories and let $\Cat^{\B}$ be the functor category of 
$\B$-diagrams in $\Cat$. We shall refer to an object in $\Cat^{\B}$ as a \emph{$\B$-category}. Recall that the Grothendieck construction $\B\!\int\! X$ on a $\B$-category $X$ is a category with objects 
$(\bld n, \bld x)$ given by an object $\bld n$ in $\B$ and an object $\bld x$ in the category $X(\bld n)$. A morphism $(\alpha,s)\colon 
(\bld m,\bld x)\to (\bld n,\bld y)$ is a morphism $\alpha\colon \bld m\to \bld n$ in $\B$ together with a morphism $s\colon X(\alpha)(\bld x)\to\bld y$ in $X(\bld n)$. The composition of morphisms is defined by 
\[
(\beta,t)\circ (\alpha,s) =(\beta\circ \alpha, t\circ X(\beta)(s)).
\]
This construction defines a functor $\B\!\int\colon \Cat^{\B}\to \Cat$ in the obvious way. We think of 
$\B\!\int \! X$ as the homotopy colimit of $X$ in $\Cat$. This is justified by Thomason's homotopy colimit theorem \cite[Theorem~1.2]{T} which states that there is a natural weak equivalence
\begin{equation}\label{eq:Thomason-equivalence}
\eta\colon \hocolim_{\bld n\in \B} \nerve(X(\bld n)) \xr{\simeq} \textstyle \nerve(\B\!\int\!X).
\end{equation}
Let us say that a functor $Y\to Y'$ between small categories is a weak equivalence if the induced map of nerves $\nerve(Y)\to \nerve(Y')$ is a weak equivalence of simplicial sets. We say that a map of $\B$-categories $X\to X'$ is a \emph{$\B$-equivalence} if the map of Grothendieck constructions $\B\!\int\!X\to \B\!\int\!X'$ is a weak equivalence in this sense. By the natural weak equivalence in \eqref{eq:Thomason-equivalence} this is equivalent to the level-wise nerve $\nerve(X)\to \nerve(X')$ being a $\B$-equivalence in the sense of the previous section. Let $w$ denote the class of weak equivalences in $\Cat$, and let $w_{\B}$ be the class of 
$\B$-equivalences in $\Cat^{\B}$. With the given definition of $\B$-equivalences it is not surprising that the categories $\Cat^{\B}$ and $\Cat$ become equivalent after localization with respect to these classes of equivalences. For the convenience of the reader we have collected the relevant background material on localization in Appendix~\ref{app:localization}. Let us write $\Delta\colon \Cat\to \Cat^{\B}$ for the functor that takes a small category to the corresponding constant $\B$-category. 

\begin{proposition}\label{prop:Bint-Delta-equivalence}
The functors $\B\!\int$ and $\Delta$ induce an equivalence of the localized categories
\[
\textstyle\B\!\int \colon \Cat^{\B}[w_{\B}^{-1}]\simeq \Cat[w^{-1}] :\!\Delta.
\]
\end{proposition}

For the proof of the proposition we need to introduce an auxiliary endofunctor on $\Cat^{\B}$. Let 
$(\B\down\bullet)$ be the $\B$-category defined by the comma categories $(\B\down \bld n)$. By definition, an object of $(\B\down \bld n)$ is a pair $(\bld m,\gamma)$ given by an object $\bld m$ and a morphism $\gamma\colon \bld m\to \bld n$ in $\B$. A morphism $\alpha\colon(\bld m_1,\gamma_1) \to
(\bld m_2,\gamma_2)$ is a morphism $\alpha\colon\bld m_1\to\bld m_2$ in $\B$ such that $\gamma_1=\gamma_2\circ\alpha$. Let $\pi_n\colon (\B\down\bld n)\to \B$ be the forgetful functor mapping 
$(\bld m,\gamma)$ to $\bld m$. Clearly these functors assemble to a map of $\B$-categories $\pi\colon (\B\down \bullet)\to\Delta(\B)$. Given a $\B$-category $X$, the \emph{bar resolution} $\overline X$ is the $\B$-category defined by the 
level-wise Grothendieck constructions 
\[
\textstyle\overline X(\bld n)=(\B\down \bld n)\!\int\! X\circ\pi_n.
\]
The structure maps making $\overline X$ a $\B$-category are inherited from the $\B$-category 
$(\B\down\bullet)$ in the obvious way. Our use of the term ``bar resolution'' is motivated by the analogous bar resolution for 
$\B$-spaces that we shall consider in Section~\ref{subsec:E2-rectification}. 
\begin{lemma}\label{lem:ev-equivalence}
There is a natural level-wise weak equivalence $\ev\colon \overline X\to X$.
\end{lemma}
\begin{proof}
For each $\bld n$ we define a functor 
$
\ev(\bld n)\colon \textstyle(\B\down \bld n)\!\int\! X\circ\pi_n\to X(\bld n).
$
An object in the domain has the form $((\bld m,\gamma),\bld x)$ with $(\bld m,\gamma)$ in $(\B\down \bld n)$ and $\bld x$ an object in $X(\bld m)$. We map this to the object $X(\gamma)(\bld x)$ in $X(\bld n)$. 
A morphism from $((\bld m_1,\gamma_1),\bld x_1)$ to  $((\bld m_2,\gamma_2),\bld x_2)$ amounts to a morphism $\alpha\colon (\bld m_1,\gamma_1)\to (\bld m_2,\gamma_2)$ in $(\B\down\bld n)$ together with a morphism $s\colon X(\alpha)(\bld x_1)\to \bld x_2$ in $X(\bld m_2)$. We map such a morphism to the morphism 
\[
X(\gamma_2)(s)\colon X(\gamma_1)(\bld x_1)=X(\gamma_2)(X(\alpha)(\bld x_1))\to X(\gamma_2)(\bld x_2)
\]
in $X(\bld n)$. These functors are compatible when $\bld n$ varies and give rise to the map of $\B$-categories in the lemma. To show that $\ev(\bld n)$ is a weak equivalence, we consider the canonical functor 
\[
j(\bld n)\colon X(\bld n)\to \textstyle(\B\down \bld n)\!\int\! X\circ\pi_n,\qquad \bld x\mapsto (1_n,\bld x)
\]
where $1_n$ denotes the terminal object in $(\B\down \bld n)$. Then $\ev(\bld n)\circ j(\bld n)$ is the identity functor on $X(\bld n)$ and it is easy to see that there is a natural transformation from the identity functor on $\textstyle(\B\down \bld n)\!\int\! X\circ\pi_n$ to $j(\bld n)\circ \ev(\bld n)$. Hence $j(\bld n)$ defines a homotopy inverse of $\ev(\bld n)$.
\end{proof}

\begin{lemma}\label{lem:pi-equivalence}
There is a natural $\B$-equivalence $\pi\colon \overline X\to \Delta(\B\!\int\! X)$.
\end{lemma}
\begin{proof}
For each $\bld n$ the forgetful functor $\pi_n\colon (\B\down \bld n)\to \B$ gives rise to a functor 
\[
\textstyle(\B\down \bld n)\!\int\! X\circ\pi_n\to \B\!\int\!X
\]
by mapping an object $((\bld m,\gamma),\bld x)$ to $(\bld m,\bld x)$. Letting $\bld n$ vary this defines the map of $\B$-categories in the lemma. We must show that the functor $\B\!\int\!\pi$ is a weak equivalence and for this we consider the diagram of categories
\[
\xymatrix{
\textstyle\B\!\int \!\big( (\B\down\bullet)\!\int\!X\circ\pi_{\bullet}\big) \ar[r]^-{\B\!\int\!\pi} 
\ar[d]_{\B\!\int\!\ev}
& 
\B\!\int\!\Delta(\B\!\int\! X) \ar[d]^{\cong}\\
\B\!\int\! X& \ar[l]_-{\mathrm{proj}}  \B\times (\B\!\int\!X)
}
\]
where $\mathrm{proj}$ denotes the projection away from the first variable. This diagram is not commutative but we claim that it commutes up to a natural transformation. Indeed, consider an object $(\bld n,(\bld m,\gamma),\bld x)$ with $\bld n$ in $\B$, $(\bld m,\gamma)$ an object in $(\B\down \bld n)$, and $\bld x$ an object in $X(\bld m)$. The functor $\B\!\int\!\ev$ maps this to $(\bld n,X(\gamma)(\bld x))$ whereas the other composition maps it to $(\bld m,\bld x)$. It is easy to see that the morphisms 
\[
(\gamma,\id_{X(\gamma)(\bld x)})\colon (\bld m,\bld x)\to (\bld n,X(\gamma)(\bld x))
\]
define a natural transformation between these functors. Since $\B\!\int\!\ev$ is a weak equivalence by Lemma~\ref{lem:ev-equivalence} and $\mathrm{proj}$ is a weak equivalence because $\B$ has an initial object, it follows that also $\B\!\int\!\pi$ is a weak equivalence.
\end{proof}

\begin{proof}[Proof of Proposition~\ref{prop:Bint-Delta-equivalence}]
We first observe that the localization of $\Cat$ with respect to $w$ actually exists since Thomason has realized it as the homotopy category of a suitable model structure, see \cite{Thomason_Cat}. With terminology from Appendix~\ref{app:localization}, Lemmas~\ref{lem:ev-equivalence} and \ref{lem:pi-equivalence} give a chain of natural 
$\B$-equivalences relating $\Delta(\B\!\int\! X)$ to $X$. The other composition $\B\!\int\!\Delta Y$ can be identified with the product category $\B\times Y$ which is weakly equivalent to $Y$ since $\B$ has an initial object. Hence the result follows from Proposition~\ref{prop:localization-equivalence}.  
\end{proof}

\begin{remark}
Let $(\bullet\down\B)$ denote the $\B^{\op}$-category defined by the comma categories 
$(\bld n\down \B)$. The universal property of the Grothendieck construction established in \cite[Proposition~1.3.1]{T} implies that $\B\!\int\!X$ can be identified with the coend $(\bullet\down \B)\times_{\B}X$ in $\Cat$. This in turn implies that the functor $\B\!\int$ participates as the left adjoint in an adjunction
\[
\textstyle\B\!\int\colon \Cat^{\B}\rightleftarrows \Cat : \!\Cat((\bullet\down\B),-).
\]  
The right adjoint takes a small category $Y$ to the $\B$-category for which the objects of $\Cat((\bld n\down\B),Y)$ are the functors from $(\bld n\down \B)$ to $Y$ and the morphisms are the natural transformations. However, this adjunction is not so useful for our purposes since it cannot be promoted to an adjunction between the braided monoidal structures we shall consider later.
\end{remark}

\subsection{Braided monoidal structures}
As in the case of $\B$-spaces considered in Section~\ref{subsec:braided-monoidal-B-spaces}, the braided monoidal structure of $\B$ induces a braided monoidal structure on $\Cat^{\B}$: Given $\B$-categories $X$ and $Y$, we define $X\boxtimes Y$ to be the left Kan extension of the ($\B\times \B$)-category 
\[
\B\times \B \xr{X\times Y} \Cat\times\Cat \xr{\times} \Cat
\]
along the monoidal structure map $\sqcup\colon \B\times\B\to \B$. Thus, the data specifying a map of 
$\B$-categories $X\boxtimes Y\to Z$ is equivalent to the data giving a map of ($\B\times \B$)-categories
 $X(\bld m)\times Y(\bld n)\to Z(\bld m\sqcup\bld n)$. We also have the level-wise description 
\[
X\boxtimes Y(\bld n)=\colim_{\bld n_1\sqcup\bld n_2\to\bld n} X(\bld n_1)\times Y(\bld n_2).
\]
The monoidal unit for the $\boxtimes$-product is the terminal $\B$-category $U^{\B}=\B(\bld 0,-)$. Using that $\Cat$ is cartesian closed one easily defines the coherence isomorphisms for associativity and unity required to make $\Cat^{\B}$ a monoidal category. We specify a braiding
$\br\colon X\boxtimes Y\to Y\boxtimes X$ on $\Cat^{\B}$ by requiring that the categorical analogue of the diagram \eqref{eq:B-spaces-braiding} be commutative. The following is the categorical analogue of Proposition~\ref{prop:B-spaces-braided-monoidal}.
\begin{proposition}
The category $\Cat^{\B}$ equipped with the $\boxtimes$-product, the unit $U^{\B}$, and the braiding $\br$ is a braided monoidal category. \qed
\end{proposition}

We use the term \emph{$\B$-category monoid} for a monoid in $\Cat^{\B}$. By the universal property of the 
$\boxtimes$-product, the data needed to specify the unit $U^{\B}\to A$ and the multiplication 
$\otimes\colon A\boxtimes A\to A$ on a $\B$-category monoid $A$ amounts to a unit object $\bld u$ in $A(\bld 0)$ and a map of ($\B\times \B$)-categories $\otimes\colon A(\bld m)\times A(\bld n)\to A(\bld m\sqcup \bld n)$ satisfying the usual associativity and unitality conditions. By the definition of the braiding, $A$ is commutative (that is, 
$\otimes\circ \br=\otimes$) if and only if the categorical version of the diagram \eqref{eq:commutativity-condition} is commutative. 

In order to talk about braided $\B$-category monoids we need the notion of a natural transformation between maps of $\B$-categories: Given maps of $\B$-categories $f,g\colon X\to Y$, a natural transformation $\phi\colon f\Rightarrow g$ is a family of natural transformations 
$\phi(\bld n)\colon f(\bld n) \Rightarrow g(\bld n)$ such that for any morphism $\alpha\colon \bld m\to\bld n$
in $\B$ we have an equality of natural transformations 
$\phi(\bld n)\circ X(\alpha) = Y(\alpha)\circ \phi(\bld m)$ between the functors $f(\bld n)\circ X(\alpha)=Y(\alpha)\circ f(\bld m)$ and $g(\bld n)\circ X(\alpha)=Y(\alpha)\circ g(\bld m)$.
Here the symbol $\circ$ denotes the usual ``horizontal'' composition, and we use the notation $X(\alpha)$ and $Y(\alpha)$ both for the functors defined by $X$ and $Y$ and for the corresponding identity natural transformations.
 A braiding of a $\B$-category monoid $A$ is then a natural transformation $\Theta\colon \otimes\Rightarrow \otimes\circ\br$ as depicted in the diagram
\begin{equation}\label{eq:Theta-braiding}
\xymatrix{
A\boxtimes A \ar[rr]^{\br} \ar[dr]_{\otimes}^{\ \ \ \ \ \Theta \ \Rightarrow}& & A\boxtimes A \ar[dl]^{\otimes}\\
& A, &
}
\end{equation}
such that $\Theta$ has an inverse and the familiar axioms for a braided monoidal structure holds. In order to formulate this in a convenient manner we observe that the data defining a natural isomorphism $\Theta$ as above amounts to a natural isomorphism
\[
\Theta_{\bld m,\bld n}\colon \bld a\otimes \bld b\to A(\chi_{\bld m,\bld n}^{-1})(\bld b\otimes \bld a)
\]
of functors $A(\bld m)\times A(\bld n)\to A(\bld m\sqcup\bld n)$ for all $(\bld m,\bld n)$, with the requirement that for each pair of morphisms $\alpha\colon \bld m_1\to \bld m_2$ and $\beta\colon \bld n_1\to\bld n_2$ we have
\[
A(\alpha\sqcup \beta)\circ \Theta_{\bld m_1,\bld n_1} =\Theta_{\bld m_2,\bld n_2}\circ(A(\alpha)\times A(\beta))
\] 
as an equality of natural transformations.

\begin{definition}\label{def:B-cat-braiding}
A braiding of a $\B$-category monoid $A$ is a natural isomorphism $\Theta$ as in \eqref{eq:Theta-braiding} such that the diagrams 
\[
\xymatrix{
\bld a\otimes \bld b\otimes\bld c \ar[r]^-{\Theta_{\bld l,\bld m}\otimes \id_{\bld c}} \ar[d]_{\Theta_{\bld l,\bld m\sqcup\bld n}}& A(\chi_{\bld l,\bld m}^{-1}\sqcup 1_{\bld n})(\bld b\otimes\bld a\otimes\bld c)
\ar[d]^{A(\chi_{\bld l,\bld m}^{-1}\sqcup 1_{\bld n})(\id_{\bld b}\otimes\Theta_{\bld l,\bld n})} \\
A(\chi^{-1}_{\bld l,\bld m\sqcup\bld n})(\bld b\otimes\bld c\otimes\bld a) \ar@{=}[r] & 
A(\chi_{\bld l,\bld m}^{-1}\sqcup 1_{\bld n})A(1_{\bld m}\sqcup\chi_{\bld l,\bld n}^{-1})
(\bld b\otimes\bld c\otimes\bld a)
}
\]
and 
\[
\xymatrix{
\bld a\otimes \bld b\otimes\bld c \ar[r]^-{\id_{\bld a}\otimes \Theta_{\bld m,\bld n}} \ar[d]_{\Theta_{\bld l\sqcup \bld m,\bld n}}& A(1_{\bld l}\sqcup\chi_{\bld m,\bld n}^{-1})(\bld a\otimes\bld c\otimes\bld b)
\ar[d]^{A(1_{\bld l}\sqcup\chi_{\bld m,\bld n}^{-1})(\Theta_{\bld l,\bld n}\otimes\id_{\bld b})} \\
A(\chi^{-1}_{\bld l\sqcup\bld m,\bld n})(\bld c\otimes\bld a\otimes\bld b) \ar@{=}[r] & 
A(1_{\bld l}\sqcup \chi_{\bld m,\bld n}^{-1})A(\chi_{\bld l,\bld n}^{-1}\sqcup 1_{\bld m})
(\bld c\otimes\bld a\otimes\bld b)
}
\]
commute for all objects $\bld a\in A(\bld l)$, $\bld b\in A(\bld m)$, and $\bld c\in A(\bld n)$.
\end{definition}

Notice that for $A$ a constant $\B$-category monoid this definition recovers the usual notion of a braided strict monoidal category. We write $\Br$-$\Cat^{\B}$ for the category of braided $\B$-category monoids and braiding preserving (strict) maps of $\B$-category monoids. Thus, a morphism 
$f\colon A\to B$ in $\Br$-$\Cat^{\B}$ is a map of $\B$-category monoids such that for all $\bld m,\bld n$ we have 
\[
f(\bld m\sqcup\bld n)\circ \Theta_{\bld m,\bld n}^A=
\Theta_{\bld m,\bld n}^B\circ(f(\bld m)\times f(\bld n))
\]
as an equality of natural transformations between functors from $A(\bld m)\times A(\bld n)$ to 
$B(\bld m\sqcup\bld n)$.  
Similarly, we write $\Br$-$\Cat$ for the category of braided strict monoidal small categories and braiding preserving strict monoidal functors. 

\begin{remark}
The natural transformations between maps of $\B$-categories make $\Cat^{\B}$ a 2-category in the obvious way. Furthermore, this enrichment is compatible with the $\boxtimes$-product such that $\Cat^{\B}$ is a braided monoidal 2-category in the sense of \cite[Section~5]{JS}. In such a setting there is a notion of braided monoidal objects with coherence isomorphisms generalizing those for a braided monoidal category. With the terminology from \cite{JS}, our notion of a braided $\B$-category monoid is thus the same thing as a braided strict monoidal object in $\Cat^{\B}$. We shall not be concerned with the coherence theory for $\B$-categories and leave the details for the interested reader.
\end{remark}

Our main goal in this subsection is to show that the functor $\B\!\int\!$ induces an equivalence between the categories $\Br$-$\Cat^{\B}$ and $\Br$-$\Cat$ after localization as in 
Proposition~\ref{prop:Bint-Delta-equivalence}. Consider in general a $\B$-category monoid $A$. Then $\B\!\int\! A$ inherits the structure of a strict monoidal category with product $\otimes\colon \B\!\int\! A\times \B\!\int\! A\to \B\!\int\! A$ defined on objects and morphisms by
\[
\begin{aligned}
&\big[(\bld m_1,\bld a_1)\xr{(\alpha,s)}(\bld m_2,\bld a_2)\big] \otimes
\big[(\bld n_1,\bld b_1)\xr{(\beta,t)}(\bld n_2,\bld b_2)\big]\\
&=\big[(\bld m_1\sqcup \bld n_1,\bld a_1\otimes \bld b_1)\xr{(\alpha\sqcup \beta,s\otimes t)}
(\bld m_2\sqcup \bld n_2,\bld a_2\otimes \bld b_2)\big].
\end{aligned}
\]
The monoidal unit for $\otimes$ is the object $(\bld 0,\bld u)$ defined by the unit object $\bld u\in A(\bld 0)$. 
Now suppose that $A$ has a braiding given by a compatible family of natural isomorphisms
$
\Theta_{\bld m,\bld n}\colon \bld a\otimes\bld b\to A(\chi_{\bld m,\bld n}^{-1})(\bld b\otimes\bld a)
$.
Then we define a braiding of $\B\!\int\! A$ by the natural transformation
\[
(\bld m,\bld a)\otimes(\bld n,\bld b)=(\bld m\sqcup \bld n,\bld a\otimes\bld b)
\xr{(\chi_{\bld m,\bld n},A(\chi_{\bld m,\bld n})(\Theta_{\bld m,\bld n}))}
(\bld n\sqcup \bld m,\bld b\otimes\bld a)=(\bld n,\bld b)\otimes(\bld m,\bld a).
\]
We summarize the construction in the next proposition.
\begin{proposition}\label{prop:Grothendieck-promotion}
The Grothendieck construction gives rise to a functor 
\[
\textstyle\B\!\int\!\colon \text{$\Br$-$\Cat^{\B}$}\to \text{$\Br$-$\Cat$}.
\eqno\qed
\] 
\end{proposition} 
\begin{remark}
It is clear from the definition that the functor $\B\!\int\!$ is monoidal and hence takes monoids in $\Cat^{\B}$ to monoids in $\Cat$. However, $\B\!\int\!$ is not braided monoidal and consequently does not take commutative monoids to commutative monoids. The main point of the above proposition is that it nonetheless preserves braided monoidal structures.  
\end{remark}

For the next proposition we write $w$ for the class of morphisms in $\Br$-$\Cat$ whose underlying functors are weak equivalences in $\Cat$. Similarly we write $w_{\B}$ for the class of morphisms in $\Br$-$\Cat^{\B}$ whose underlying maps of $\B$-categories are $\B$-equivalences. 

\begin{proposition}\label{prop:Bint-Delta--braided-equivalence}
The functors $\B\!\int$ and $\Delta$ induce an equivalence of the localized categories
\[
\textstyle\B\!\int \colon \text{$\Br$-$\Cat^{\B}$}[w^{-1}_{\B}]\simeq 
\text{$\Br$-$\Cat$}[w^{-1}] :\!\Delta.
\]
\end{proposition}

The proof of the proposition is based on the following lemma.

\begin{lemma}\label{lem:B-down-bullet-braided}
The bar resolution functor taking a $\B$-category $X$ to $\overline X$ can be promoted to an endofunctor on $\Br$-$\Cat^{\B}$.
\end{lemma}
\begin{proof}
Consider in general a $\B$-category monoid $A$ with unit object $\bld u\in A(\bld 0)$ and multiplication specified by functors $\otimes\colon A(\bld m)\times A(\bld n)\to A(\bld m\sqcup\bld n)$. Then  
$\overline A$ inherits a monoid structure with unit object 
$(1_{\bld 0},\bld u)$ in $\overline A(\bld 0)$, and  multiplication 
\[
\textstyle\overline\otimes\colon (\B\down\bld m)\!\int\!A \circ \pi_{\bld m}\times
(\B\down\bld n)\!\int\!A \circ \pi_{\bld n}\to
(\B\down\bld m\sqcup \bld n)\!\int\!A \circ \pi_{\bld m\sqcup\bld n}
\]
defined on objects and morphisms by
\[
\begin{aligned}
&\big[((\bld m_1,\gamma_1),\bld a_1)\xr{(\alpha,s)} ((\bld m_2,\gamma_2),\bld a_2)\big] \overline\otimes
\big[((\bld n_1,\delta_1),\bld b_1)\xr{(\beta,t)} ((\bld  n_2,\delta_2),\bld b_2)\big] \\
&=\big[((\bld m_1\sqcup \bld n_1,\gamma_1\sqcup\delta_1),\bld a_1\otimes\bld b_1)\xr{(\alpha\sqcup \beta,s\otimes t)} ((\bld m_2\sqcup \bld n_2,\gamma_2\sqcup\delta_2),\bld a_2\otimes \bld b_2)\big].
\end{aligned}
\]
Now suppose that in addition $A$ has a braiding specified by a family of natural isomorphisms
$
\Theta_{\bld m,\bld n}\colon \bld a\otimes\bld b\to A(\chi_{\bld m,\bld n}^{-1})(\bld b\otimes\bld a)
$. Then we define a braiding $\overline\Theta$ of $\overline A$ by the natural isomorphisms 
\[
\xymatrix{
((\bld m,\gamma),\bld a)\overline\otimes ((\bld n,\delta),\bld b) \ar[rrr]^-{\overline\Theta}\ar@{=}[d] &&& 
\overline A(\chi_{\bld m,\bld n}^{-1})\big[((\bld n,\delta),\bld b)\overline\otimes ((\bld m,\gamma),\bld a)\big]
\ar@{=}[d] \\
((\bld m\sqcup\bld n,\gamma\sqcup\delta),\bld a\otimes\bld b) 
\ar[rrr]^-{(\chi_{\bld m,\bld n},A(\chi_{\bld m,\bld n})(\Theta_{\bld m,\bld n}))\\ } && &
((\bld n\sqcup\bld m,\chi_{\bld m,\bld n}^{-1}\circ(\delta\sqcup\gamma)),\bld b\otimes\bld a).
}
\]
It is straight forward to check the axioms for a braiding as formulated in Definition~\ref{def:B-cat-braiding}.
\end{proof}

\begin{proof}[Proof of Proposition~\ref{prop:Bint-Delta--braided-equivalence}]
We first observe that the work of Fiedorowicz-Stelzer-Vogt \cite{FSV} shows that the localization of $\Br$-$\Cat$ exists, 
cf.\  Example~\ref{ex:FSV-example} in the appendix. Given this, the proof of the proposition follows the same pattern as the proof of Proposition~\ref{prop:Bint-Delta-equivalence}: For a braided $\B$-category monoid $A$ we know from Lemma~\ref{lem:B-down-bullet-braided} that  
$\overline A$ has the structure of a braided $\B$-category monoid and it is clear from the definitions that the $\B$-equivalences $\ev$ and $\pi$ in Lemmas~\ref{lem:ev-equivalence} and \ref{lem:pi-equivalence} are morphisms in $\Br$-$\Cat^{\B}$. With the terminology from 
Appendix~\ref{app:localization} we therefore have a chain of natural $\B$-equivalences in $\Br$-$\Cat^{\B}$ relating $A$ and 
$\Delta(\B\!\int\! A)$. Given a braided strict monoidal category $\cA$, the other composition 
$\B\!\int\!\Delta(\cA)$ can be identified with the product category $\B\times \cA$ as an object in $\Br$-$\Cat$. Clearly the projection $\B\times \cA\to \cA$ is a weak equivalence in $\Br$-$\Cat$ and the proposition therefore follows from Proposition~\ref{prop:localization-equivalence}. 
\end{proof}

\subsection{Rectification and strict commutativity}\label{subsec:rectification}
Now we proceed to introduce the $\B$-category rectification functor and show how it allows us to replace braided monoidal structures by strictly commutative structures up to $\B$-equivalence.
Let $(\mathcal{A}, \otimes, \bld u)$ be a braided strict monoidal small category. 
We shall define the \emph{$\B$-category rectification} of $\cA$ to be a certain $\B$-category $\Phi(\cA)$ such that the objects of $\Phi(\cA)(\bld n)$ are $n$-tuples $(\bld a_1,\ldots,\bld a_n)$ of objects in $\mathcal{A}$. By definition $\Phi(\cA)(\bld 0)$ has the ``empty string'' $\emptyset$ as its only object. The morphisms in $\Phi(\cA)(\bld n)$ are given by  
\[
\Phi(\mathcal{A})(\bld{n})\big((\bld a_1,\ldots,\bld a_n),(\bld b_1,\ldots,\bld b_n)\big)=
\mathcal{A}(\bld a_1\otimes\ldots\otimes\bld a_n,\bld b_1\otimes\ldots\otimes\bld b_n)
\]
with composition inherited from $\cA$. Here we agree that the $\otimes$-product of the empty string is the unit object $\bld u$ so that $\Phi(\cA)(\bld 0)$ can be identified with the monoid of endomorphisms 
$\cA(\bld u,\bld u)$. For a morphism $\alpha\colon \bld m\to\bld n$ in $\B$, the induced functor 
\[
\Phi(\cA)(\alpha)\colon \Phi(\cA)(\bld m)\to \Phi(\cA)(\bld n)
\]
is given on objects by
\[
\Phi(\cA)(\alpha)(\bld a_1,\dots,\bld a_m)=(\bld a_{\bar\alpha^{-1}(1)},\dots,\bld a_{\bar\alpha^{-1}(n)})
\]
where $\bar\alpha\colon\bld m\to\bld n$ denotes the underlying injection, $\bld a_{\bar\alpha^{-1}(j)}=\bld a_i$ if $\bar\alpha(i)=j$, and $\bld a_{\bar\alpha^{-1}(j)}=\bld u$ if $j$ is not in the image of $\bar\alpha$. In order to describe the action on morphisms we use Lemma~\ref{lemma unique decomposition} to get a factorization $\alpha=\Upsilon(\nu)\circ \xi$ with $\nu\in \cM(\bld m,\bld n)$ and $\xi\in \cB_m$. The action of $\Phi(\cA)(\alpha)$ on a morphism $f$ from $(\bld a_1,\dots,\bld a_m)$ to $(\bld b_1,\dots,\bld b_m)$
is then determined by the commutativity of the diagram
\[
\xymatrix{
\bld a_{\bar\alpha^{-1}(1)}\otimes\dots\otimes\bld a_{\bar\alpha^{-1}(n)} \ar[d]_{\Phi(\cA)(\alpha)(f)}
\ar@{=}[r] & \bld a_{\bar\xi^{-1}(1)}\otimes\dots\otimes\bld a_{\bar\xi^{-1}(m)}
& \bld a_1\otimes\dots\otimes \bld a_m \ar[l]_-{\xi_*}\ar[d]_f\\
\bld b_{\bar\alpha^{-1}(1)}\otimes\dots\otimes\bld b_{\bar\alpha^{-1}(n)}
\ar@{=}[r] & \bld b_{\bar\xi^{-1}(1)}\otimes\dots\otimes\bld b_{\bar\xi^{-1}(m)}
& \bld b_1\otimes\dots\otimes \bld b_m \ar[l]_-{\xi_*}
}
\]
where $\xi_*$ denotes the canonical action of $\xi$ on the $m$-fold $\otimes$-product. In particular, this describes the action of $\Phi(\cA)(\alpha)$ on the generating morphisms in  
Lemma~\ref{lemma generating morphisms for B} and one easily checks that the relations in this lemma are preserved. Hence the above construction does indeed define a $\B$-category. 
The construction is clearly functorial in $\cA$ so that we have defined a functor 
$\Phi\colon \text{$\Br$-$\Cat$}\to \Cat^\B$. This functor was first considered in the unpublished Master's Thesis by the second author \cite{mir}. 

The $\B$-category $\Phi(\cA)$ is homotopy constant in positive degrees in the sense of the next lemma. Here we let $\B_+$ denote the full subcategory of $\B$ obtained by excluding the initial object $\bld 0$. 

\begin{lemma}\label{lem:Phi-positive-equivalences}
The functor $\Phi(\cA)(\alpha)\colon \Phi(\cA)(\bld m)\to \Phi(\cA)(\bld n)$ is a weak equivalence for any morphism $\alpha\colon\bld m\to\bld n$ in $\B_+$.
\end{lemma}
\begin{proof}
We first consider a morphism of the form $j\colon \bld 1\to\bld m$ and claim that the functor $\Phi(\cA)(j)$ is in fact an equivalence of categories. Indeed, let $p\colon \Phi(\cA)(\bld m)\to \Phi(\cA)(\bld 1)$ be the obvious functor taking $(\bld a_1,\dots,\bld a_m)$ to $(\bld a_1\otimes\dots\otimes\bld a_m)$.    
Then $p\circ j$ is the identity on $\Phi(\cA)(\bld 1)$ and it is clear that the other composition $j\circ p$ is naturally isomorphic to the identity on 
$\Phi(\cA)(\bld m)$. For a general morphism $\alpha\colon\bld m\to\bld n$ in $\B_+$ we have a commutative diagram
\[
\xymatrix{
&\Phi(\cA)(\bld 1) \ar[dl]_{\Phi(\cA)(j)} \ar[dr]^{\Phi(\cA)(\alpha j)}& \\
\Phi(\cA)(\bld m) \ar[rr]^{\Phi(\cA)(\alpha)} && \Phi(\cA)(\bld n)
}
\]
and the result follows. 
\end{proof}

The next proposition shows that $\Phi$ takes braided monoidal structures to strictly commutative structures and is the reason why we refer to $\Phi$ as a ``rectification functor''. Let us write $\cC(\Cat^{\B})$ for the category of commutative $\B$-category monoids.  

\begin{proposition}
 The $\B$-category ${\Phi}(\mathcal{A})$ is a commutative monoid in $\Cat^\B$ and ${\Phi}$ 
defines a  functor $\Phi\colon\BrCat\to \mathcal{C}(\Cat^\B)$.
\end{proposition}
\begin{proof}
 We define  functors $\otimes\colon \Phi(\mathcal{A})(\bld{m})\times\Phi(\mathcal{A})(\bld{n})\ra \Phi(\mathcal{A})(\bld{m}\sqcup\bld{n})$ by 
\[
(\bld a_1,\dots,\bld a_m)\otimes(\bld b_1,\dots,\bld b_n)=
(\bld a_1,\dots,\bld a_m,\bld b_1,\dots,\bld b_n)
\]
on objects and by applying the monoidal structure $f\otimes g$ of $\cA$ on morphisms. These functors are natural in $(\bld m,\bld n)$ as one verifies by checking for the generating morphisms in Lemma~\ref{lemma generating morphisms for B}. By the universal property of the $\boxtimes$-product we therefore get an associative product on $\Phi(\cA)$. It is clear that the object $\emptyset$ in $\Phi(\cA)(\bld 0)$ specifies a unit for this multiplication. The categorical analogue of the criteria for commutativity expressed by the commutativity of  \eqref{eq:commutativity-condition} clearly holds on objects and on morphisms it follows from the naturality of the braiding on $\cA$.
\end{proof}

\begin{remark}\label{rem:Phi-non-strict}
The definition of $\Phi(\cA)$ can be extended to braided monoidal small categories $\cA$ that are not necessarily strict monoidal. Indeed, the objects of $\Phi(\cA)(\bld n)$ are again  $n$-tuples $(\bld a_1,\dots,\bld a_n)$ of objects in $\cA$ and a morphism
from  $(\bld a_1,\dots,\bld a_n)$ to $(\bld b_1,\dots,\bld b_n)$ is defined to be a morphism
\[
(\cdots((\bld a_1\otimes\bld a_2)\otimes \bld a_3)\otimes\dots\otimes\bld a_{n-1})\otimes\bld a_n\to
(\cdots((\bld b_1\otimes\bld b_2)\otimes \bld b_3)\otimes\dots\otimes\bld b_{n-1})\otimes\bld b_n
\]
in $\cA$. Proceeding as in the strict monoidal case, the coherence theory for braided monoidal categories ensures that $\Phi(\cA)$ canonically has the structure of a commutative $\B$-category monoid. This is functorial with respect to braided strong monoidal functors that strictly preserve the unit objects.
\end{remark}

We shall view $\cC(\Cat^{\cB})$ as the full subcategory of $\BrCatB$ given by the braided $\B$-category monoids with identity braiding $\otimes=\otimes\circ \mathfrak b$. 

\begin{proposition}\label{prop:Phi-Bint-equivalence}
The composite functor
\[
\BrCat\xr{\Phi} \cC(\Cat^{\B}) \xr{} \BrCatB\xr{\B\!\int\!}\BrCat
\]
is related to the identity functor on $\BrCat$ by a natural weak equivalence.
\end{proposition}
\begin{proof}
For a braided strict monoidal category $\cA$ we define a functor $P\colon \B\!\int\!\Phi(\cA)\to \cA$ such that $P$ takes an object $(\bld m,(\bld a_1,\dots,\bld a_m))$ to $\bld a_1\otimes\dots\otimes\bld a_m$. A morphism $(\alpha,f)$ from $(\bld m,(\bld a_1,\dots,\bld a_m))$ to $(\bld n,(\bld b_1,\dots,\bld b_n))$ is given by a morphism $\alpha\colon \bld m\to\bld n$ in $\B$ together with a morphism $f$ from
$\bld a_{\bar\alpha^{-1}(1)}\otimes\dots\otimes\bld a_{\bar\alpha^{-1}(n)}$ to $\bld b_1\otimes \dots\otimes\bld b_n$. Using Lemma~\ref{lemma unique decomposition} we get a factorization $\alpha=\Upsilon(\nu)\circ \xi$ with $\nu\in \cM(\bld m,\bld n)$ and $\xi\in \cB_m$, and let $P(\alpha,f)$ be the composition 
\[
\xymatrix{ 
\bld a_1\otimes\dots\otimes \bld a_m  \ar[r]^{P(\alpha,f)} \ar[d]_{\xi_*}& \bld b_1\otimes\dots\otimes \bld b_n \\
\bld a_{\bar\xi^{-1}(1)}\otimes\dots\otimes\bld a_{\bar\xi^{-1}(m)} \ar@{=}[r]&
\bld a_{\bar\alpha^{-1}(1)}\otimes\dots\otimes\bld a_{\bar\alpha^{-1}(n)}.\ar[u]^f
}
\]
It is straight forward to check that $P$ is a braided strict monoidal functor. Furthermore, it follows from the definition of Thomason's equivalence \eqref{eq:Thomason-equivalence} that the composition 
\[
\nerve(\Phi(\cA)(\bld 1))\to \hocolim_{\bld n\in \B}\nerve(\Phi(\cA)(\bld n))\xr{\eta} \textstyle\nerve(\B\!\int\!\Phi(\cA))\xr{P} \nerve(\cA)
\] 
is the canonical identification. Hence it suffices to prove that the first map, induced by the inclusion $\{\bld 1\}\to\B$, is a weak equivalence. To this end we first restrict $\nerve(\Phi(\cA))$ to $\B_+$ such that all the structure maps are weak equivalences by Lemma~\ref{lem:Phi-positive-equivalences}. Then it follows from \cite[Lemma~IV.5.7]{GJ} that the diagram 
\[
\xymatrix{
\nerve(\Phi(\cA)(\bld 1)) \ar[r] \ar[d]& \nerve(\Phi(\cA))_{h\B_+}\ar[d]\\
\{\bld 1\} \ar[r] &\nerve(\B_+)
}
\]  
is homotopy cartesian, and since $\nerve(\B_+)$ is contractible this in turn implies that $\nerve(\Phi(\cA)(\bld 1)) \to\nerve(\Phi(\cA))_{h\B_+}$ is a weak equivalence. Secondly, it is easy to see that the inclusion of $\B_+$ in $\B$ is homotopy cofinal such that the induced map 
$\nerve(\Phi(\cA))_{h\B_+}\to \nerve(\Phi(\cA))_{h\B}$ is a weak equivalence by \cite[Theorem~19.6.13]{hirsch}.
\end{proof}

Combining the result obtained in this section we get the following theorem.

\begin{theorem}\label{thm:Braided-B-Cat-rectification}
Every braided $\B$-category monoid is related to a strictly commutative $\B$-category monoid by a chain of natural $\B$-equivalences in $\BrCatB$.
\end{theorem}
\begin{proof}
Given a braided $\B$-category monoid $A$, we have the following chain of $\B$-equivalences
\[
\textstyle A\simeq \Delta(\B\!\int\! A)\simeq \Delta\big(\B\!\int\!\Phi(\B\!\int\!A)\big) \simeq \Phi(\B\!\int\!A).
\]
The first and last equivalences are the chains of $\B$-equivalences $\Delta(\B\!\int\!(-))\simeq(-)$ from the proof of Proposition~\ref{prop:Bint-Delta--braided-equivalence} and the 
$\B$-equivalence in the middle is obtained by applying $\Delta$ to the weak equivalence $\B\!\int\!\Phi(-)\simeq(-)$ in 
Proposition~\ref{prop:Phi-Bint-equivalence}.
\end{proof}

\section{$E_2$ spaces and braided commutativity}\label{sec:E2-spaces}
Building on the categorical foundations in the last section, we proceed to show that every $E_2$ space can be represented by a strictly commutative $\B$-space monoid up to $\B$-equivalence.

\subsection{Operadic interpretation of braided monoidal structures}\label{subsec:operads}
In order to relate our results from the previous section to multiplicative structures on spaces, it is convenient to work with an operadic interpretation of braided monoidal structures. By a $\Cat$-operad we understand an operad internal to the category $\Cat$. Thus, a $\Cat$-operad $\sM$ is given by a sequence of small categories $\sM(k)$ for $k\geq 0$ together with functors 
\[
\gamma \colon \sM(k)\times \sM(j_1)\times\dots\times \sM(j_k)\to \sM(j_1+\dots j_k),
\]
a unit object $\bld 1\in \sM(1)$, and a right $\Sigma_k$-action on $\sM(k)$. These data are required to satisfy the usual axioms for associativity, unity, and equivariance as listed in \cite[Definition~1.1]{may}. We shall always assume that a $\Cat$-operad $\sM$ is reduced in the sense that $\sM(0)$ is the terminal category with one object and one morphism. A $\Cat$-operad as above gives rise to a monad $\mathbb M$ on $\Cat$ by letting 
\[
\mathbb M(X)=\coprod_{n\geq 0}\sM(k)\times_{\Sigma_k}X^k
\]    
for a small category $X$. Here $X^0$ denotes the terminal category. By definition, an $\sM$-algebra in $\Cat$ is an algebra for this monad and we write $\sM$-$\Cat$ for the category of $\sM$-algebras. An algebra structure $\theta\colon \mathbb M(X)\to X$ is determined by a family of functors 
$\theta_k\colon \sM(k)\times X^k\to X$ satisfying the axioms listed in \cite[Lemma~1.4]{may}.

Following \cite[Section~ 8]{FSV} we introduce a $\Cat$-operad $\Br$ such that $\Br$-algebras are braided strict monoidal small categories. The objects of $\Br(k)$ are the elements $a\in \Sigma_k$ and given objects $a$ and $b$, a morphism $\alpha\colon a\to b$ is an element $\alpha\in \cB_k$ such that 
$\bar\alpha a=b$. Composition in $\Br(k)$ is inherited from $\cB_k$ and the right action of an element $g\in \Sigma_k$ is defined on objects and morphisms by taking $(\alpha\colon a\to b)$ to $(\alpha\colon ag\to bg)$. The structure map 
\[
\gamma\colon \Br(k)\times \Br(j_1)\times\dots\times\Br(j_k)\to \Br(j_1+\dots+j_k)
\]
is defined on objects by 
\[
\gamma(a,b_1,\dots,b_k)=a(j_1,\dots,j_k)\circ b_1\sqcup\dots\sqcup b_k
\]
where $a(j_1,\dots,j_k)$ denotes the block permutation of $\bld j_1\sqcup\dots\sqcup \bld j_k$ specified by $a$. The action on morphisms is analogous except for the obvious permutation of the indices. 
Let $\sA$ be the discrete $\Cat$-operad given by the objects of $\Br$. It is well-known and easy to check that $\sA$-algebras are the same thing as monoids in $\Cat$, that is, strict monoidal small categories. Hence a $\Br$-algebra $X$ has an underlying strict monoidal category with unit object determined by the structure map $\theta_0\colon \Br(0)\times X^0\to X$ and monoidal structure $\otimes=\theta_2(1_2,-,-)$ determined by restricting the structure map $\theta_2\colon\Br(2)\times X^2\to X$ to the unit object $1_2\in \Br(2)$. With $t$ the non-unit object of $\Br(2)$ and $\zeta$ the generator of 
$\cB_2$, the morphism $\zeta\colon 1_2\to t$ determines a natural transformation 
\[
\theta_2(\zeta,\id_{\bld x_1},\id_{\bld x_2})\colon \bld x_1\otimes\bld x_2\to \bld x_2\otimes \bld x_1
\]  
which gives a braiding of $X$. Conversely, for a braided strict monoidal category $X$ we define a $\Br$-algebra structure by the functors
$
\theta_k\colon \Br(k)\times X^k\to X
$
taking a tuple of morphisms $\alpha\colon a\to b$ in $\Br(k)$ and $f_i\colon \bld x_i\to\bld y_i$ in $X$ for $i=1,\dots,k$, to the composition in the commutative diagram
\[
\xymatrix{
\bld x_{a^{-1}(1)}\otimes\dots\otimes \bld x_{a^{-1}(k)} \ar[r]^{\alpha_*}  
\ar[d]_{f_{a^{-1}(1)}\otimes \dots\otimes f_{a^{-1}(k)}} &
\bld x_{b^{-1}(1)}\otimes\dots\otimes \bld x_{b^{-1}(k)}
\ar[d]^{f_{b^{-1}(1)}\otimes \dots\otimes f_{b^{-1}(k)}} \\
\bld y_{a^{-1}(1)}\otimes\dots\otimes \bld y_{a^{-1}(k)} \ar[r]^{\alpha_*} &
 \bld y_{b^{-1}(1)}\otimes\dots\otimes \bld y_{b^{-1}(k)}.
}
\]  
Here $\alpha_*$ denotes the canonical action of $\alpha$ defined by the braided monoidal structure. Summarizing, we have the following consistency result that justifies our use of the notation $\Br$-$\Cat$ in the previous section.

\begin{lemma}\label{lem:Cat-Br-algebras}
The category  $\BrCat$ of $\Br$-algebras is isomorphic to the category of braided strict monoidal categories. \qed
\end{lemma}

It is natural to ask for an analogous operadic characterization of braided $\B$-category monoids. However, since the symmetric groups do not act on the iterated $\boxtimes$-products in $\Cat^{\B}$, we instead have to work with braided operads as introduced by Fiedorowicz \cite{fied}. By definition, a braided $\Cat$-operad $\sM$ is a sequence of small categories $\sM(k)$ for $k\geq 0$ together with structure maps and a unit just as for a $\Cat$-operad. The difference from an (unbraided) $\Cat$-operad is that in the braided case we require a right $\cB_k$-action on $\sM(k)$ for all $k$ such that the braided analogue of the equivariance axiom for a $\Cat$-operad holds. A braided $\Cat$-operad $\sM$ defines a monad on $\Cat^{\B}$ by letting 
\[
\mathbb M(X)=\coprod_{k\geq 0}\sM(k)\times_{\cB_k}X^{\boxtimes k}
\] 
for a $\B$-category $X$. By definition, an $\sM$-algebra in $\Cat^{\B}$ is an algebra for this monad and we write $\sM$-$\Cat^{\B}$ for the category of 
$\sM$-algebras. It follows from the universal property of the $\boxtimes$-product that an $\sM$-algebra structure on a $\B$-category $X$ can be described in terms of functors
\begin{equation}\label{eq:level-Br-algebra}
\theta_k\colon\sM(k)\times X(\bld n_1)\times\dots\times X(\bld n_k)\to X(\bld n_1\sqcup\dots\sqcup\bld n_k)
\end{equation}
such that the usual associativity and unity axioms hold as well as the equivariance axiom stating that the diagram
\[
\xymatrix{
\sM(k)\times X(\bld n_1)\times\dots\times X(\bld n_k) \ar[r]^-{\theta_k\circ(\sigma\times\id)} \ar[d]_{\id\times\bar\sigma}& 
X(\bld n_1\sqcup\dots\sqcup\bld n_k)\ar[d]^{X(\sigma(n_1,\dots,n_k))} \\ 
\sM(k)\times X(\bld n_{\bar \sigma^{-1}(1)})\times \dots\times X(\bld n_{\bar\sigma^{-1}(k)}) \ar[r]^-{\theta_k} & 
X(\bld n_{\bar \sigma^{-1}(1)}\sqcup\dots \sqcup\bld n_{\bar\sigma^{-1}(k)})
}
\]
is commutative for all $\sigma\in \cB_k$. We also use the notation $\Br$ for the braided $\Cat$-operad for which the category $\Br(k)$ has objects the 
elements $a\in \cB_k$ and a morphism $\alpha\colon a\to b$ is an element $\alpha\in \cB_k$ such that $\alpha a=b$. The structure maps making this a braided $\Cat$-operad are defined as for the analogous unbraided operad. Let $\sA$ be the discrete braided $\Cat$-operad given by the objects in $\Br$. It is easy to see that an $\sA$-algebra in $\Cat^{\B}$ is the same thing as a $\B$-category monoid and hence that a $\Br$-algebra is a $\B$-category monoid with extra structure. Indeed, suppose that $X$ is a $\Br$-algebra in $\Cat^{\B}$ and write $\otimes\colon X\boxtimes X\to X$ for the monoid structure defined by restricting $\theta_2\colon \Br(2)\times X^{\boxtimes 2}\to X$ to the unit object $1_2\in \Br(2)$. With $\zeta$ the standard generator of $\cB_2$, the morphism $\zeta\colon 1_2\to \zeta$ determines a natural isomorphism $\Theta=\theta(\zeta,-,-)$ as in the diagram \eqref{eq:Theta-braiding} and $\Theta$ satisfies the axioms for a braiding of $X$. Arguing as in the unbraided setting we get the following analogue of Lemma~\ref{lem:Cat-Br-algebras}.

\begin{lemma}\label{lem:Cat-B-Br-algebras}
The category $\BrCatB$ is isomorphic to the category of braided $\B$-category monoids.\qed
\end{lemma}  

\subsection{Rectification of $E_2$ algebras}\label{subsec:E2-rectification}
Applying the nerve functor $\nerve$ to the unbraided $\Cat$-operad $\Br$ we get an operad $\nerve\Br$ in simplicial sets with $k$th space $\nerve\Br(k)$. This is an $E_2$ operad in the sense that its geometric realization is equivalent to the little 2-cubes operad, cf.\ \cite[Proposition~8.13]{FSV}. Since the nerve functor preserves products it is clear that it induces a functor $\nerve\colon \BrCat\to \text{$\nerve\Br$-$\cS$}$. This was first observed by Fiedorowicz~\cite{fied}, and is the braided version of the analogous construction for permutative categories considered by May~\cite{may2}. Similarly, the braided version of the $\Cat$-operad $\Br$ gives rise to the braided operad 
$\nerve\Br$ in simplicial sets. By the level-wise characterization of $\Br$-algebras in \eqref{eq:level-Br-algebra} it is equally clear that the level-wise nerve induces a functor $\nerve\colon \BrCatB\to\text{$\nerve\Br$-$\cS^{\B}$}$.

Now we want to say that the homotopy colimit functor induces a functor from $\nerve\Br$-$\cS^{\B}$ to $\nerve\Br$-$\cS$, but to explain this properly requires some preparation. Recall that the pure braid group $\cP_k$ is the kernel of the projection $\varPi\colon\cB_k\to \Sigma_k$. Following \cite{FSV}, a braided operad $\sM$ can be ``debraided'' to an (unbraided) operad $\sM/\cP_k$ with $k$th term the orbit space $\sM(k)/\cP_k$. The structure maps are inherited from the structure maps of $\sM$ and $\Sigma_k$ acts from the right via the isomorphism $\Sigma_k\cong \cB_k/\cP_k$. For instance, the debraiding of the braided $\Cat$-operad $\Br$ is the corresponding unbraided $\Cat$-operad $\Br$ and similarly for the braided operad $\nerve\Br$. In the following lemma we consider the product of the latter with an arbitrary braided operad $\sM$ and form the debraided operad $(\nerve\Br\times \sM)/\cP$.

\begin{lemma}\label{lem:hocolim-promoted}
Let $\sM$ be a braided operad in simplicial sets. Then the homotopy colimit functor can be promoted to a functor 
\[
(-)_{h\B}\colon \text{$\sM$-$\cS^{\B}$}\to \text{$(\nerve\Br\times \sM)/\cP$-$\cS$}.
\]
\end{lemma}
\begin{proof}
Let $X$ be a $\B$-space with $\sM$-action defined by natural maps
\[
\theta_k\colon \sM(k)\times X(\bld n_1)\times\dots\times X(\bld n_k)\to X(\bld n_1\sqcup\dots\sqcup\bld n_k).
\] 
To $X$ we associate the simplicial category (that is, simplicial object in $\Cat$) $\B\!\int\!X$ obtained by applying the Grothendieck construction in each simplicial degree of $X$ thought of as a $\B$-diagram of simplicial discrete categories. It is clear from the definition that the nerve of $\B\!\int\!X$ can be identified with $X_{h\B}$. Let us further view $\Br(k)$ as a constant simplicial category and $\sM(k)$ as a simplicial discrete category. Then we define maps of simplicial categories
\[
\textstyle\theta_k\colon \Br(k)\times \sM(k)\times (\B\!\int\!X)^k\to \B\!\int\!X
\]
such that a tuple of objects $a\in\Br(k)$, $\mathsf m\in \sM(k)$, and $(\bld m_i,x_i)\in \B\!\int\!X$ for $i=1,\dots,k$, is mapped to the object
\[
\big(\bld m_{\bar a^{-1}(1)}\sqcup\dots\sqcup \bld m_{\bar a^{-1}(k)},X(a(m_1,\dots,m_k))\theta_k(\mathsf m,x_1,\dots,x_k)\big).
\]
A tuple of morphisms $\alpha\colon a\to b$ in $\Br(k)$ and $\beta_i\colon (\bld m_i,x_i)\to (\bld n_i,y_i)$ in $\B\!\int\!X$ for $i=1,\dots,k$, is mapped to the morphism specified by
\[
\alpha(m_{\bar\alpha^{-1}(1)},\dots,m_{\bar\alpha^{-1}(k)})\circ \beta_{\bar\alpha^{-1}(1)}\sqcup\dots\sqcup\beta_{\bar\alpha^{-1}(k)}.
\]
Evaluating the nerves of these simplicial categories we get a map of bisimplicial sets and by restricting to the simplicial diagonal a map of simplicial sets 
\[
\textstyle(\nerve\Br(k)\times \sM(k))\times \nerve(\B\!\int\!X)^k\to \nerve(\B\!\int\!X).
\]
It is not difficult to check that these maps satisfy the conditions for a braided operad action and hence descends to an action of the debraided operad 
$(\nerve\Br\times \sM)/\cP$. Clearly this is functorial in $X$.
\end{proof}

When $\sM$ is the braided operad $\nerve \Br$ we can compose with the diagonal map of (unbraided) operads 
$ \nerve\Br/\cP\to(\nerve\Br\times \nerve\Br)/\cP$ to get the next lemma.

\begin{lemma}\label{lem:hocolim-Br-promoted}
The homotopy colimit functor can be promoted to a functor
\[
(-)_{h\B}\colon \text{$\nerve\Br$-$\cS^{\B}$}\to \text{$\nerve\Br$-$\cS$}.
\eqno\qed
\]
\end{lemma}
The natural maps introduced so far are compatible in the expected way.
\begin{proposition}\label{prop:Thomason-diagram}
The diagram
\[
\xymatrix{
\BrCatB \ar[r]^{\nerve} \ar[d]_{\B\!\int} & \text{$\nerve\Br$-$\cS^{\B}$} \ar[d]^{(-)_{h\B}}\\
\BrCat \ar[r]^{\nerve} & \text{$\nerve\Br$-$\cS$}
}
\]
commutes up to natural weak equivalence. 
\end{proposition}
\begin{proof}
Given a braided $\B$-category $X$, we claim that  Thomason's equivalence $\eta$ in \eqref{eq:Thomason-equivalence} is in fact a morphism in $\text{$\nerve\Br$-$\cS$}$.
In order to verify the claim we first use Proposition~\ref{prop:Grothendieck-promotion} and 
Lemma~\ref{lem:Cat-Br-algebras} to get an explicit description of the $\nerve\Br$-algebra structure on $\nerve(\B\!\int\!X)$. Secondly, we use Lemmas~\ref{lem:Cat-B-Br-algebras} and 
\ref{lem:hocolim-Br-promoted} to get an explicit description of the $\nerve\Br$-algebra structure on 
$(\nerve X)_{h\B}$. It is then straight forward (although somewhat tedious) to check that Thomason's explicit description of $\eta$ in \cite[Lemma~1.2.1]{T} is compatible with the algebra structures.
\end{proof}

We proceed to show that the functor $(-)_{h\B}$ in Lemma~\ref{lem:hocolim-Br-promoted} induces an equivalence after suitable localizations of the domain and target. Let us write $w$ for the class of morphisms in \text{$\nerve\Br$-$\cS$} whose underlying maps of spaces are weak equivalences and $w_{\B}$ for the class of morphisms in \text{$\nerve\Br$-$\cS^{\B}$} whose underlying maps of $\B$-spaces are $\B$-equivalences. The following is the $\B$-space version of Proposition~\ref{prop:Bint-Delta--braided-equivalence}. As usual $\Delta$ denotes the constant functor embedding. 

\begin{proposition}\label{prop:hocolim-B-equivalence}
The functors $(-)_{h\B}$ and $\Delta$ induce an equivalence of the localized categories
\[
(-)_{h\B}\colon \text{$\nerve\Br$-$\cS^{\B}$}[w_{\B}^{-1}]\simeq \text{$\nerve\Br$-$\cS$}[w^{-1}]:\!\Delta.
\]
\end{proposition}
 For the proof of the proposition we need to invoke the \emph{bar resolution} for $\B$-spaces. Given a $\B$-space $X$, this is the $\B$-space $\overline X$ defined by 
\[
\overline X(\bld n)=\hocolim_{(\B\!\downarrow \bld n)}X\circ\pi_{\bld n}
\] 
with notation as for the categorical bar resolution considered in Section~\ref{subsec:B-categories-Grothendieck}. (See e.g.\ \cite{Hollender-V_modules} for the interpretation of this as an actual bar construction.) Arguing as in the proof of Lemma~\ref{lem:hocolim-promoted} one sees that this construction can be promoted to an endofunctor on $\text{$\nerve\Br$-$\cS^{\B}$}$. 

\begin{proof}[Proof of Proposition~\ref{prop:hocolim-B-equivalence}]
First recall that the localization $\text{$\nerve\Br$-$\cS$}[w^{-1}]$ exists since it can be realized as the homotopy category of a suitable model structure. As for the categorical analogue in Proposition~\ref{prop:Bint-Delta--braided-equivalence} there are natural  $\B$-equivalences $\ev\colon \overline A\to A$ and $\pi\colon \overline A\to \Delta(A_{h\B})$ in $\nerve\Br$-$\cS^{\B}$. For a $\Br$-algebra $Y$ in $\cS$, the other composition 
$\Delta(Y)_{h\B}$ can be identified with the product algebra $\nerve\B\times Y$ such that the projection defines a weak equivalence of $\Br$-algebras $\Delta(Y)_{h\B}\xr{\sim} Y$. The statement therefore follows from Proposition~\ref{prop:localization-equivalence}.
\end{proof}

With these preparations we can finally prove that $\nerve\Br$-algebras in $\cS^{\B}$ can be rectified to strictly commutative $\B$-space monoids. Our proof of this result differs from the proof of the analogous categorical statement in Theorem~\ref{thm:Braided-B-Cat-rectification} since we do not have a space-level version of the rectification functor $\Phi$. Instead we shall make use of the functor $F\colon \text{$\nerve\Br$-$\cS$}\to \text{$\Br$-$\Cat$}$ introduced by 
Fiedorowicz-Stelzer-Vogt \cite{FSV} and then compose the latter with $\Phi$. The relevant facts about the functor $F$ are discussed in the context of localization in Example~\ref{ex:FSV-example}.  

\begin{theorem}\label{thm:Br-SB-rectification}
Every $\nerve\Br$-algebra in $\cS^{\B}$ is related to a strictly commutative $\B$-space monoid by a chain of natural $\B$-equivalences in 
$\nerve\Br$-$\cS^{\B}$.
\end{theorem}
\begin{proof}
Let $A$ be an $\nerve\Br$-algebra in $\cS^{\B}$. Then $A_{h\B}$ is an $\nerve\Br$-algebra in $\cS$ and applying the functor $F$ we get a $\Br$-algebra $F(A_{h\B})$ in $\Cat$. We claim that $A$ is related to the commutative $\B$-space monoid $\nerve\Phi(F(A_{h\B}))$ by a chain of $\B$-equivalences in $\nerve\Br$-$\cS^{\B}$. To this end we first proceed as in the proof of Proposition~\ref{prop:hocolim-B-equivalence} to get a chain of $\B$-equivalences $A\simeq \Delta(A_{h\B})$. Then we compose the chains of weak equivalences 
\[
\textstyle A_{h\B}\simeq \nerve F(A_{h\B})\simeq \nerve(\B\!\int\!\Phi(F(A_{h\B})))\simeq \nerve\Phi(F(A_{h\B}))_{h\B}
\]
defined respectively in \cite[C.2]{FSV}, Proposition~\ref{prop:Phi-Bint-equivalence}, and Proposition~\ref{prop:Thomason-diagram}. This in turn gives a chain of $\B$-equivalences 
\[
A\simeq\Delta(A_{h\B})\simeq  \Delta(\nerve\Phi(F(A_{h\B}))_{h\B})\simeq \nerve\Phi(F(A_{h\B})),
\]
again by Proposition~\ref{prop:hocolim-B-equivalence}.
\end{proof}

\begin{example}\label{ex:X-bullet-example-2}
In general an (unbraided) operad $\mathsf M$ in $\cS$ gives rise to a functor $\mathsf M\colon \cI^{\op}\to \cS$ as explained in \cite{CMT}. Given a based space $X$ we have the $\cI$-space $X^{\bullet}$ from Example~\ref{ex:X-bullet-example-1} and may form the coend $\mathsf M\otimes_{\cI}X^{\bullet}$ (whose geometric realization is denoted $\mathsf M|X|$ by May~\cite{may}). In the same way a braided operad $\mathsf M$ gives rise to a functor $\mathsf M\colon \B^{\op}\to\cS$ and using the same notation for the pullback of $X^{\bullet}$ to a $\B$-space we may form the coend $\mathsf M\otimes_{\B}X^{\bullet}$ considered by Fiedorowicz~\cite{fied}. Writing $\mathsf M/\mathcal P$ for the debraided operad, the fact that the pure braid groups $\mathcal P_n$ act trivially on $X^n$ implies that there is a natural isomorphism $\mathsf M\otimes_{\B}X^{\bullet}\cong \mathsf M/\mathcal P_n\otimes_{\cI}X^{\bullet}$. Now specialize to the braided operad $\nerve\Br$ and recall that the homotopy colimit $X^{\bullet}_{h\B}$ can be identified with the coend $\nerve(\bullet\!\downarrow\!\B)\otimes_{\B}X^{\bullet}$. Proceeding as in \cite[Section~4.2]{SHom} we define a map of $\B^{\op}$-spaces $\nerve(\bullet\!\downarrow\!\B)\to \nerve\Br$ such that the induced map of coends
\[
\nerve(\bullet\!\downarrow\!\B)\otimes_{\B}X^{\bullet} \to \nerve\Br\otimes_{\B}X^{\bullet}
\] 
is an equivalence. The above remarks together with the fact that the geometric realization of the debraiding $\nerve\Br/\mathcal P$ is equivalent to the little 2-cubes operad $\mathcal C_2$ imply that there are equivalences
\[
|\nerve\Br\otimes_{\B}X^{\bullet}|\cong |\nerve\Br/\mathcal P\otimes_{\cI}X^{\bullet}|
\simeq \mathcal C_2\otimes_{\cI}|X|^{\bullet}.
\]
For connected $X$ it therefore follows from \cite[Theorem~2.7]{may}
that the geometric realization of $X^{\bullet}_{h\B}$ is homotopy equivalent to $\Omega^2\Sigma^2(|X|)$. We may interpret this as saying that the commutative $\B$-space monoid $X^{\bullet}$ represents the 2-fold loop space $\Omega^2\Sigma^2(|X|)$.
\end{example}

\section{Classifying spaces for braided monoidal categories}\label{sec:classifying}
We consider a monoidal category $(\cA, \otimes, I)$ and therein a monoid $A$, a right $A$-module $M$, and a left $A$-module $N$.
Suppressing a choice of parentheses from the notation, the two-sided bar construction $B^{^{_\otimes}}_\bullet (M,A,N)$ is the simplicial object defined by
\begin{equation*}
 [k] \mapsto M\otimes A^{\otimes k}\otimes N
\end{equation*}
with structure maps as for the usual bar construction for spaces, see for instance \cite[Chapter~9]{may}. If the unit $I$ for the monoidal structure is 
both a right and left $A$-module we can define the bar construction on $A$ as 
$B^{^{_\otimes}}_\bullet(A)=B^{^{_\otimes}}_\bullet (I,A,I)$. This works in 
particular when $I$ is a terminal object in $\cA$. 

In order to say something about the multiplicative properties 
of $B^{^{_\otimes}}_\bullet(A)$ we investigate how monoids behave with respect to the monoidal product. 
If $\cA$ is a braided monoidal category with braiding $b$ the monoidal product $A\otimes B$ of two monoids 
$A$ and $B$ is again a monoid. Suppressing parentheses, the multiplication $\mu_{A\otimes B}$ is the morphism 
$$A\otimes B\otimes A\otimes B \xrightarrow{\id_A\otimes b_{B,A}\otimes \id_B} A\otimes A\otimes B\otimes B
\xrightarrow{\mu_A\otimes\mu_B} A\otimes B$$
where $\mu_A$ and $\mu_B$ are the multiplications of the monoids $A$ and $B$ respectively. Unlike in a symmetric 
monoidal category, the monoidal product of two commutative monoids in $\cA$ is not necessarily a commutative monoid. 
But it is straightforward to check that if $A$ is a commutative monoid, then the multiplication $\mu_A\co A\otimes A\ra A$ 
is a monoid morphism. 
Suppose given a monoid $A$ in $\cA$ such that the unit $I$ is a right and left $A$-module. Then the above implies that for each
$k$, $B^{^{_\otimes}}_k(A)$ is a monoid. If  in addition $A$ is commutative, the family of multiplication maps assemble into 
a morphism $B^{^{_\otimes}}_\bullet(A)\otimes B^{^{_\otimes}}_\bullet(A)\ra B^{^{_\otimes}}_\bullet(A)$ of simplicial objects, 
where the monoidal product is taken degreewise. The bar construction on a commutative 
monoid $A$ is a simplicial monoid in $\cA$ with this multiplication.

Now we specialize to the braided monoidal category $\cS^\B$ of $\B$-spaces. Here we can realize a 
simplicial object $Z_\bullet$ by taking the diagonal $|Z_\bullet|$ of the two simplicial 
directions to obtain a $\B$-space. We define the bar construction on a 
$\B$-space monoid $A$ as $\BBox(A)= |\BBox_\bullet (A)|$. 
From now on we will refer to the simplicial version 
as the simplicial bar construction. The above discussion about the multiplicative 
properties of the simplicial bar 
construction implies the following result.

\begin{lemma}
 The bar construction $\BBox(A)$ on a commutative $\B$-space monoid $A$ is a
 (not necessarily commutative) monoid in $\cS^\B$.
  \qed
\end{lemma}

Recall that the natural transformation $\nu_{A,B}\co A\hB \times B\hB \ra (A\boxtimes B)\hB$ from Lemma~\ref{lemma X and Y flat implies weak equivalnece between times and boxtimes}
gives the homotopy colimit functor $(-)\hB\co \SB \ra\cS$ the structure of a 
lax monoidal functor. As usual when we have a lax monoidal functor, 
it follows that if $A$ is a $\B$-space monoid, then $A\hB$ inherits the structure of a monoid in 
$\cS$. If $M$ is a right $A$-module, then $M\hB$ inherits the structure of a right 
$A\hB$-module in $\cS$ and similarly for a left $A$-module $N$. We can then apply the two-sided 
simplicial bar construction in $\cS$ to $A\hB$, $M\hB$ and $N\hB$ and obtain 
$B_\bullet(M\hB, A\hB, N\hB)$. The natural transformation  $\nu$ gives rise 
to maps $B_k(M\hB, A\hB, N\hB)\rightarrow \BBox_k(M, A, N)\hB$ that commute with the 
simplicial structure maps. Hence we obtain a morphism 
$$B(M\hB, A\hB, N\hB)\rightarrow \BBox(M, A, N)\hB$$ 
in $\cS$. By specializing to the case where $M$ and $N$ is the unit $U^{\B}$ we can relate 
$\BBox(A)\hB$ to $B(A\hB)$ via $B(U^{\B}\hB, A\hB, U^{\B}\hB)$. 
The homotopy colimit of $U^{\B}$ is homeomorphic to $\nerve\B$ which is a contractible simplicial set. 
Hence the map 
$B(\nerve\B, A\hB, \nerve\B) \ra B(A\hB)$ induced by the
projection $\nerve\B\ra \ast$ is a weak equivalence.

\begin{proposition}\label{proposition bar construction and homotopy colimit}
 If $A$  is a $\B$-space monoid with underlying flat $\B$-space the above defined maps
$$\BBox(A)\hB \xleftarrow{\simeq} B(\nerve\B, A\hB, \nerve\B) \xrightarrow{\simeq} B(A\hB)$$
are weak equivalences.
\end{proposition}

\begin{proof}
The argument for the right hand map being a weak equivalence is given before the proposition. 
The map  $(A\hB)^{\times k}\ra (A^{\boxtimes k})\hB$ is a weak equivalence for 
each $k\geq0$ since $A$ is flat, 
see Lemma~\ref{lemma X and Y flat implies weak equivalnece between times and boxtimes}.
It follows that the left hand map is the diagonal  of a map of bisimplicial sets which is a
weak equivalence at each simplicial degree of the bar construction. Therefore it is itself 
a weak equivalence. 
\end{proof}

Our goal is to use the bar construction in $\B$-spaces to give a double delooping of the group completion of $A\hB$ 
for a commutative $\B$-space monoid $A$ with underlying flat $\B$-space. In order to apply the previous proposition twice we will show that the bar construction on something flat is also flat.

\begin{lemma}\label{lemma bar construction flat}
 If $A$  is a $\B$-space monoid with underlying flat $\B$-space, then the underlying 
 $\B$-space of the bar construction 
 $\BBox(A)$ on $A$ is also flat. 
\end{lemma}

\begin{proof}
When $A$ is flat, it follows from Lemma~\ref{lemma pushout-product axiom} that $\BBox_k(A)$ is flat for each $k\geq0$. 
The criterion for flatness given in Proposition~\ref{prop flat B-space} 
can be checked in each simplicial degree. Thus, $\BBox(A)$ is the diagonal of 
a bisimplicial object which is flat at each simplicial degree of the bar construction 
and is therefore flat.
\end{proof}

We use the well known fact that the group completion of a homotopy commutative simplicial monoid $M$ may be modelled by the canonical map $M\rightarrow \Omega (B (M)^{\mathrm{fib}})$, where the fibrant replacement 
$B(M)^{\mathrm{fib}}$ is the singular simplicial set of the geometric realization of $B(M)$. 
By a double delooping of a simplicial set $K$ we mean a based simplicial set $L$ such that $\Omega^2(L^{\mathrm{fib}})\simeq K$.

\begin{proposition}
If $A$ is a commutative $\B$-space monoid with underlying flat $\B$-space, 
then $\BBox(\BBox(A))\hB$ is a double delooping of the group completion of $A\hB$.
\end{proposition}

\begin{proof}
Letting $A$ equal $\BBox(A)$ in Proposition~\ref{proposition bar construction and homotopy colimit} 
and using 
 Lemma~\ref{lemma bar construction flat} we get 
 \begin{equation*}
\BBox(\BBox(A))\hB \simeq  B( \BBox(A)\hB).
\end{equation*}
Evaluating $\Omega((-)^{\mathrm{fib}})$ on this we get equivalences
\begin{equation*}
\Omega(\BBox(\BBox(A))\hB^{\mathrm{fib}}) \simeq \Omega( B( \BBox(A)\hB)^{\mathrm{fib}})
 \simeq  \BBox (A)\hB^{\mathrm{fib}} \simeq B(A\hB)^{\mathrm{fib}}
\end{equation*}
where the map in the middle is an equivalence  since $\BBox (A)\hB$ is connected and hence group-like. Looping once more we see that $\BBox(\BBox(A))\hB$ is indeed a double delooping of the group completion of $B(A\hB)$.
\end{proof}

Recall from Remark~\ref{rem:Phi-non-strict} that we can construct a commutative $\B$-space monoid $\nerve\Phi(\cA)$ for any  braided (not necessarily strict) monoidal small category. 
Next, we show that  $\nerve\Phi(\cA)$ has underlying flat $\B$-space so we can apply the above result to the double bar construction on $\nerve\Phi(\cA)$.

\begin{lemma}
Let $\cA$ is a  braided  monoidal small category. The commutative $\B$-space monoid $\nerve\Phi(\cA)$ has underlying flat $\B$-space. 
\end{lemma}

\begin{proof}
 Here we prove the result for a braided strict monoidal small  category, the non-strict case is left to the reader. We use the criterion given in Proposition~\ref{prop flat B-space}.  For each braided injection $\bld m\ra \bld n$ the induced functor $\Phi(\cA)(\bld m) \ra \Phi(\cA)(\bld n)$ is injective on both objects and morphisms. Thus the nerve of that map is a cofibration of simplicial sets. 
The functor $\Phi(\cA)(\bld m)\ra \Phi(\cA)(\bld m\sqcup \bld n)$ induced by the inclusion of $\bld m$ in $\bld m\sqcup \bld n$ takes an object $(\bld a_1,\ldots, \bld a_m)$ to 
$(\bld a_1,\ldots, \bld a_m, U^{\B}, \ldots U^{\B})$. Since we have a strict monoidal structure it takes a morphism $f$ to the morphism $f\boxtimes \id_{U^{\B}\boxtimes \cdots \boxtimes U^{\B}}= f$. 
If we consider a diagram similar to \eqref{diagram condition for flatness} for the $\B$-category $\Phi(\cA)$ it is clear that the intersection of the images of $\Phi(\cA)(\bld{l\sqcup m})$ and $\Phi(\cA)(\bld{m\sqcup n})$ 
in $ \Phi(\cA)(\bld{l\sqcup m\sqcup n})$ equals the image of $\Phi(\cA)(\bld m)$. The same then holds for the $\B$-space $\nerve\Phi(\cA)$.
\end{proof}

\begin{corollary}
If $\cA$ is a  braided  monoidal small category, then 
$\BBox(\BBox(\mathrm{N}\Phi(\mathcal{A})))\hB$ is a double 
delooping of the group completion of $\nerve\cA$.
\end{corollary}

\begin{proof}
The underlying $\B$-space of $\nerve\Phi(\cA)$ is flat, so we can apply the proposition and get that $\BBox(\BBox(\mathrm{N}\Phi(\mathcal{A})))\hB$ is a double 
delooping of the group completion of $\nerve\Phi(\cA)\hB$. But by combining Propositions~\ref{prop:Thomason-diagram} and \ref{prop:Phi-Bint-equivalence}, the latter is weakly equivalent to $\nerve\cA$. 
\end{proof}

\section{$\mathcal{I}$-categories and $E_{\infty}$ spaces}\label{sec:I-section}
In this section we focus on diagrams indexed by the category $\cI$ and we record the constructions and results analogous to those worked out for diagrams indexed by the category $\B$ in the previous sections.
The proofs are completely analogous to those in the braided case (if not simpler) and will be omitted throughout. We then relate this material to the category of symmetric spectra. 

Let $\Cat^{\cI}$ denote the category of $\cI$-categories with the symmetric monoidal convolution product inherited from $\cI$. The Grothendieck construction defines a functor $\cI\!\int\colon \Cat^{\cI}\to \Cat$ and a map of $\cI$-categories $X\to Y$ is said to be an 
\emph{$\cI$-equivalence} if the induced functor $\cI\!\int\! X\to\cI\!\int\! Y$ is a weak equivalence. We write $\Sym$ for the symmetric monoidal analogue of the $\Cat$-operad $\Br$. Thus, the category $\Sym(k)$ has as its objects the elements $a$ in $\Sigma_k$ and a morphism 
$\alpha\colon a\to b$ is an element $\alpha\in \Sigma_k$ such that $\alpha a=b$. It is proved in \cite{may2} that a $\Sym$-algebra in $\Cat$ is the same thing as a permutative (i.e., symmetric strict monoidal) category and that the nerve $\nerve\Sym$ can be identified with the Barratt-Eccles operad. The latter is an $E_{\infty}$ operad in the sense that $\nerve\Sym(k)$ is $\Sigma_k$-free and contractible for all $k$. As in Proposition~\ref{prop:Bint-Delta--braided-equivalence} one checks that there is an equivalence of localized categories 
\begin{equation*}
\textstyle\cI\!\int \colon \text{$\Sym$-$\Cat^{\cI}$}[w_{\cI}^{-1}] \simeq \text{$\Sym$-$\Cat$}[w^{-1}] :\!\Delta.
\end{equation*}
The rectification functor $\Phi$ from Section~\ref{subsec:rectification} also has a symmetric monoidal version, now in the form of a functor $\Phi\colon \text{$\Sym$-$\Cat^{\cI}$}\to \cC(\Cat^{\cI})$ where the codomain is the category of commutative $\cI$-category monoids. The composite functor 
\begin{equation}\label{eq:I-rectification-equivalence}
\text{$\Sym$-$\Cat$}\xr{\Phi} \cC(\Cat^{\cI}) \xr{} \text{$\Sym$-$\Cat^{\cI}$}\xr{\cI\!\int} \text{$\Sym$-$\Cat$}
\end{equation} 
is weakly equivalent to the identity functor and arguing as in the proof of Theorem~\ref{thm:Braided-B-Cat-rectification} we get the following result.
\begin{theorem}
Every $\Sym$-algebra in $\Cat^{\cI}$ is related to a strictly commutative $\cI$-category monoid by a chain of $\cI$-equivalences in 
\text{$\Sym$-$\Cat^{\cI}$}.\qed
\end{theorem}

In particular, every symmetric monoidal category is weakly equivalent to one of the form $\cI\!\int\! A$ for $A$ a strictly commutative $\cI$-category monoid. Now let $\cS^{\cI}$ be the category of $\cI$-spaces equipped with the symmetric monoidal convolution product inherited from $\cI$. A map of $\cI$-spaces $X\to Y$ is an \emph{$\cI$-equivalence} if the induced map of homotopy colimits $X_{h\cI}\to Y_{h\cI}$ is a weak equivalence and the $\cI$-space version of Proposition~\ref{prop:hocolim-B-equivalence} gives an equivalence of the localized categories
\begin{equation}\label{NSym-I-equivalence}
(-)_{h\cI}\colon\text{$\nerve\Sym$-$\cS^{\cI}$}[w_{\cI}^{-1}]\simeq \text{$\nerve\Sym$-$\cS$}[w^{-1}] :\!\Delta.
\end{equation}  
Furthermore, one checks that the $\cI$-category version of Thomason's equivalence \eqref{eq:Thomason-equivalence} gives a natural weak equivalence relating the two compositions in the diagram
\begin{equation}\label{eq:I-Thomason-commutativity}
\xymatrix{
\text{$\Sym$-$\Cat^{\cI}$} \ar[r]^{\nerve} \ar[d]_{\cI\!\int} &\text{$\nerve\Sym$-$\cS^{\cI}$} \ar[d]^{(-)_{h\cI}}\\
\text{$\Sym$-$\Cat$} \ar[r]^{\nerve}  &\text{$\nerve\Sym$-$\cS$}.
}
\end{equation}
Arguing as in the proof of Theorem~\ref{thm:Br-SB-rectification} one can use this to show that every $\nerve\Sym$-algebra in $\cS^{\cI}$ is $\cI$-equivalent to one that is strictly commutative. However, a stronger form of this statement has been proved in \cite{Sagave-Schlichtkrull}: There is a model structure on 
$\text{$\nerve\Sym$-$\cS^{\cI}$}$ such that the equivalence \eqref{NSym-I-equivalence} can be derived from a Quillen equivalence, and a further model structure on 
$\cC(\cS^{\cI})$ (the category of commutative $\cI$-space monoids) making the latter Quillen equivalent to $\text{$\nerve\Sym$-$\cS^{\cI}$}$.

\subsection{Symmetric spectra and $E_{\infty}$ spaces} Let $\Sp^{\Sigma}$ be the category of symmetric spectra as defined in \cite{HSS}. The smash product of symmetric spectra makes this a symmetric monoidal category with monoidal unit the sphere spectrum. Given an (unbased) space $X$ we write $\Sigma^{\infty}(X_+)$ for the suspension spectrum with $n$th space $X_+\wedge S^n$ where $X_+$ denotes the union of $X$ with a disjoint base point. If $X$ is an $E_{\infty}$ space (i.e., an algebra for an $E_{\infty}$ operad in $\cS$), then $\Sigma^{\infty}(X_+)$ is an $E_{\infty}$ symmetric ring spectrum for the same operad. It is proved in \cite{Elmendorf-M_infinite-loop} that in general an $E_{\infty}$ symmetric ring spectrum is stably equivalent to a strictly commutative symmetric ring spectrum. However, the proof of this fact is not very constructive and it is of interest to find more memorable 
commutative models of the $E_{\infty}$ ring spectra in common use. Here we shall do this for $E_{\infty}$ symmetric ring spectra of the form $\Sigma^{\infty}(\nerve \cA_+)$ for a permutative category $\cA$. The relevant operad is the Barratt-Eccles operad $\nerve\Sym$ as explained above. In order to make use of the rectification functor $\Phi$ we recall from \cite[Section~3]{Sagave-Schlichtkrull} that the suspension spectrum functor extends to a strong symmetric monoidal functor $\mathbb S^{\cI}\colon \cS^{\cI}\to \Sp^{\Sigma}$ taking an $\cI$-space $X$ to the symmetric spectrum $\mathbb S^{\cI}[X]$ with $n$th space $X(\bld n)_+\wedge S^n$. Given a permutative category $\cA$ we may apply this functor to the commutative $\cI$-space monoid $\nerve\Phi(\cA)$ to get the commutative symmetric ring spectrum $\mathbb S^{\cI}[\nerve\Phi(\cA)]$.

\begin{proposition}\label{prop:S[NA]-model}
Given a permutative category $\cA$, the commutative symmetric ring spectrum 
$\mathbb S^{\cI}[\nerve\Phi(\cA)]$ is related to 
$\Sigma^{\infty}(\nerve \cA_+)$  by a chain of natural stable equivalences of $E_{\infty}$ symmetric ring spectra.
\end{proposition}
\begin{proof}
Composing the natural weak equivalence relating the composite functor \eqref{eq:I-rectification-equivalence} to the identity functor with Thomason's weak equivalence relating the two compositions in \eqref{eq:I-Thomason-commutativity}, we get a weak equivalence of $\nerve\Sym$-algebras 
\[
\textstyle \nerve\cA \xl{\simeq} \nerve(\cI\!\int\!\Phi(\cA)) \xl{\simeq} (\nerve\Phi(\cA))_{h\cI}.
\] 
Furthermore, using the bar resolution $\overline{(-)}$ as in Section~\ref{subsec:E2-rectification} we get a chain of $\cI$-equivalences
\[
\Delta(\nerve\Phi(\cA)_{h\cI}) \xl{\simeq} \overline{\nerve\Phi(\cA)}\xr{\simeq} \nerve\Phi(\cA)
\]
in $\text{$\nerve\Sym$-$\cS^{\cI}$}$. This gives the result since the functor $\mathbb S^{\cI}[-]$ 
takes $\cI$-equivalences to stable equivalences.
\end{proof}

We also note that the symmetric spectrum $\mathbb S^{\cI}[\nerve\Phi(\cA)]$ has several of the pleasant properties discussed in \cite[Section~5]{HSS}: The fact that it is $S$-cofibrant (this is what some authors call flat) ensures that it is homotopically well-behaved with respect to the smash product, and the fact that it is semistable ensures that its spectrum homotopy groups can be identified with the stable homotopy groups of $\nerve \cA$.

\begin{example}
The underlying infinite loop space $Q(S^0)$ of the sphere spectrum plays a fundamental role in stable homotopy theory. In order to realize the $E_{\infty}$ ring spectrum $\Sigma^{\infty}(Q(S^0)_+)$ as a commutative symmetric ring spectrum, we use that $Q(S^0)$ is weakly equivalent to the classifying space of Quillen's localization construction $\Sigma^{-1}\Sigma$, where as usual $\Sigma$ denotes the category of finite sets and bijections. (We refer to \cite{Grayson-higher} for a general discussion of Quillen's localization construction and to \cite{Sagave-Schlichtkrull} for an explicit description of $\Sigma^{-1}\Sigma$.) The category $\Sigma^{-1}\Sigma$ inherits a permutative structure from $\Sigma$ and it follows from Proposition~\ref{prop:S[NA]-model} that the commutative symmetric ring spectrum $\mathbb S^{\cI}[\nerve\Phi(\Sigma^{-1}\Sigma)]$ is a model of $\Sigma^{\infty}(Q(S^0)_+)$.  
\end{example}

\appendix
\section{Localization of categories}\label{app:localization}

We make some elementary remarks on localization of categories. Let $\cC$ be a (not necessarily small) category and let $\cV$ be a class of morphisms in $\cC$. Recall that a \emph{localization of $\cC$ with respect to $\cV$} is a category $\cC'$ together with a functor $L\colon C\to \cC'$ that maps the morphisms in $\cV$ to isomorphisms in $\cC'$ and is initial with this property: Given a category $\cD$ and a functor $F\colon \cC\to \cD$ that maps the morphisms in $\cV$ to isomorphisms in $\cD$, there exists a unique functor $F'\colon \cC'\to \cD$ such that $F=F'\circ L$. Clearly a localization of $\cC$ with respect to $\cV$ is uniquely determined up to isomorphism if it exists. We sometimes use the notation $\cC\to \cC[\cV^{-1}]$ for such a localization. It will be convenient to assume that $\cC$ and $\cC[\cV^{-1}]$ always have the same objects. 

Let again $\cC$ be a category equipped with a class of morphisms $\cV$ and consider a category $\cA$ together with a pair of functors $F,G\colon \cA\to \cC$. In this situation we say that $F$ and $G$ are \emph{related by a chain of natural transformations in $\cV$}, written $F\simeq_{\cV}G$, if there exists a finite sequence of functors $H_1,\dots, H_n$ from $\cA$ to $\cC$ with $H_1=F$ and $H_n=G$, and for each $1\leq i<n$ either a natural transformation $H_i\to H_{i+1}$ with values in $\cV$ or a natural transformation $H_{i+1}\to H_i$ with values in $\cV$. In the next proposition we write $I_{\cC}$ for the identity functor on $\cC$.

\begin{proposition}\label{prop:localization-equivalence}
Let $\cC$ and $\cD$ be categories related by the functors $F\colon \cC\to \cD$ and $G\colon \cD\to \cC$.
Suppose that $\cV$ is a class of morphisms in $\cC$ and that $\cW$ is a class of morphisms in $\cD$ such that $F(\cV)\subseteq \cW$, $G(\cW)\subseteq \cV$, $G\circ F\simeq_{\cV} I_{\cC}$, and $F\circ G\simeq_{\cW} I_{\cD}$. Then the localization of $\cC$ with respect to $\cV$ exists if and only if the localization of $\cD$ with respect to $\cW$ exists, and in this case $F$ and $G$ induce an equivalence of categories $F\colon \cC[\cV^{-1}]\rightleftarrows \cD[\cW^{-1}] :\!G$.
\end{proposition}
\begin{proof}
Suppose that a localization of $\cD$ with respect to $\cW$ exists and write $L\colon\cD\to \cD'$ for such a localization. Then we define a category $\cC'$ with the same objects as $\cC$ and morphism sets 
$\cC'(C_1,C_2)=\cD'(LF(C_1),LF(C_2))$. The composition in $\cC'$ is inherited from the composition in 
$\cD'$. We claim that the canonical functor $L^{\cC}\colon \cC\to \cC'$, which is the identity on objects and given by $LF$ on morphisms, is a localization of $\cC$ with respect to $\cV$. Thus, consider a category $\cE$ and a functor $H\colon \cC\to \cE$ that maps the morphisms in $\cV$ to isomorphisms in $\cE$. We must define a functor $H'\colon \cC'\to \cE$ such that $H=H'\circ L^{\cC}$, and it is clear that we must have $H'(C)=H(C)$ for all objects $C$ in $\cC'$. In order to define the action on morphisms, we factor the composite functor 
$K=H\circ G\colon \cD\to \cE$ over the localization of $\cD'$ to get a functor $K'\colon\cD'\to \cE$. The relation $G\circ F\simeq_{\cV}I_{\cC}$ gives a natural isomorphism $\phi_C\colon HGF(C)\to H(C)$ and we define the action of $H'$ on the morphism set $\cC'(C_1,C_2)$ to be the composition
\[
\cD'(LF(C_1),LF(C_2))\xr{K'}\cE(HGF(C_1),HGF(C_2))\xr{\phi_{C_2}\circ(-)\circ \phi_{C_1}^{-1}}
\cE(H(C_1),H(C_2)).
\] 
It is immediate from the definition that $H'$ satisfies the required conditions and it remains to show that it is uniquely determined. The composition $L^{\cC}\circ G$ factors over $\cD'$ to give the left hand square in the commutative diagram
\[
\xymatrix@-1pt{
\cD\ar[r]^G \ar[d]_L &  \cC \ar[r]^F \ar[d]_{L^{\cC}}& \cD \ar[d]_L\\
\cD' \ar[r]^{G'} & \cC' \ar[r] & \cD'.
}
\]
Notice that the relation $F\circ G\simeq_{\cW} I_{\cD}$ gives a natural isomorphism relating the composition $L\circ F\circ G$ to $L$. Let $\mathcal J$ be the category with objects $0$ and $1$, and two non-identity arrows $i\colon 0\to 1$ and $j\colon 1\to 0$. (Thus, $\mathcal J$ is a groupoid with inverse isomorphisms $i$ and $j$.) Then we may interpret the natural isomorphism in question as a functor $\mathcal D\times \mathcal J\to \mathcal D'$, or, by adjointness, a functor $\mathcal D\to (\mathcal D')^{\mathcal J}$. The latter factors over $\cD'$ to give a natural isomorphism relating the composition in the bottom row of the diagram to the identity functor on $\cD'$. It follows that $G'\colon \cD'\to \cC'$ is fully faithful and consequently that $H'$ is uniquely determined on the full subcategory of $\cC'$ generated by objects of the form $G(D)$ for $D$ in $\cD$. Furthermore, the relation $G\circ F\simeq_{\cV}I_{\cC}$ implies that any morphism in $\cC'$ can be written as a composition of morphisms in this subcategory with morphisms in the image of $L^{\cC}$ and inverses of morphisms in the image of 
$L^{\cC}$. This shows that $H'$ is uniquely determined on the whole category $\cC'$. The last statement in the proposition is an immediate consequence. 
\end{proof}

\begin{example}\label{ex:FSV-example}
The work of Fiedorowicz-Stelzer-Vogt \cite{FSV} fits into this framework. Let $\sM$ be a $\Cat$-operad that is $\Sigma$-free in the sense that $\Sigma_k$ acts freely on $\sM(k)$ for all $k$. In \cite[C.2]{FSV} the authors define a functor $F\colon\text{$\nerve \sM$-$\cS$}\to 
\text{$\sM$-$\Cat$}$ and show that if $\sM$ satisfies a certain ``factorization condition'', then there are chains of weak equivalences $\nerve\circ F\simeq_wI$ and $F\circ \nerve \simeq_wI$ with $I$ the respective identity functors. By \cite[Lemma~8.12]{FSV} this applies in particular to the (unbraided) operad $\Br$ discussed in Section~\ref{subsec:operads}. It is well-known that the localization of $\nerve\Br$-$\cS$ with respect to the weak equivalences exists, since it can be realized as the homotopy category of a suitable model category. Thus, it follows from Proposition~\ref{prop:localization-equivalence} that also the localization of $\Br$-$\Cat$ with respect to the weak equivalences exists and that these localized categories are equivalent. This is shown in \cite[Proposition~7.4]{FSV} except that the discussion of Grothendieck universes and ``localization up to equivalence'' is not really needed in order to state this result.
\end{example}

\bibliographystyle{alpha}

\end{document}